\documentclass[twocolumn,amsthm,secthm]{autart}

\usepackage{graphicx}
\usepackage{amsmath}
\usepackage{amssymb}
\usepackage{bm}
\usepackage{enumitem}
\usepackage{placeins}
\usepackage{booktabs}

\usepackage{xcolor}
\usepackage{tikz}
\usetikzlibrary{arrows.meta,positioning,calc}
\tikzset{
  flowbox/.style={
    rectangle, rounded corners=3pt, draw=black, thick,
    fill=gray!12, text centered, font=\small,
    minimum height=1.6em, inner xsep=4pt, inner ysep=3pt
  },
  flowarrow/.style={-{Latex[length=5pt,width=4pt]}, thick, gray!60!black}
}

\numberwithin{equation}{section}
\numberwithin{figure}{section}
\numberwithin{table}{section}

\newcommand{\R}{\mathbb{R}}
\newcommand{\E}{\mathbb{E}}
\newcommand{\Var}{\mathrm{Var}}

\newcommand{\ud}{\,\mathrm{d}}

\newcommand{\xLbar}{\bar{x}^L}
\newcommand{\xFbar}{\bar{x}^F}
\newcommand{\vc}{v_c}
\newcommand{\smax}{s_{\max}}
\newcommand{\uF}{u^F}
\newcommand{\piF}{\pi^F}
\newcommand{\Mhat}{\widehat{M}}
\newcommand{\Mbar}{\overline{M}}
\newcommand{\FI}{\mathcal{I}}
\newcommand{\xFdot}{\dot{\bar{x}}^F}

\newcommand{\xFvec}{\bar{\mathbf{x}}^F}
\newcommand{\xLvec}{\bar{\mathbf{x}}^L}
\newcommand{\xFvdot}{\dot{\bar{\mathbf{x}}}^F}
\newcommand{\Jvec}{\mathbf{J}}
\newcommand{\gtil}{\tilde g}

\begin{document}

\begin{frontmatter}

\runtitle{Strategic Inference of Adversarial Navigation Objectives for UUVs}

\title{Strategic Inference of Adversarial Navigation Objectives for Unmanned Underwater Vehicles\thanksref{note}}

\thanks[note]{This paper was not presented at any IFAC meeting. This work was partially supported by the ONR grant under \#N00014-24-1-2432, the Simons Foundation (MP-TSM-00002783) and the NSF grant DMS-2420988. $^*$Corresponding author.}

\author[ucsb_math,ucsb_stats]{Ruimeng Hu}\ead{rhu@ucsb.edu},
\author[ucsb_math]{Xu Yang$^*$}\ead{xy6@ucsb.edu}

\address[ucsb_math]{Department of Mathematics, University of California, Santa Barbara, CA 93106, USA}
\address[ucsb_stats]{Department of Statistics and Applied Probability, University of California, Santa Barbara, CA 93106, USA}

\begin{keyword}
Unmanned underwater vehicles; intent inference; Fisher information.
\end{keyword}

\begin{abstract}
We study destination inference for adversarial unmanned underwater navigation in a spatially varying current field. The red vehicle is modeled as approximately following a Hamilton-Jacobi (HJ) time-optimal path toward an unknown destination, while a blue vehicle observes a noisy realization of that trajectory. We derive a continuous-time likelihood model for this problem and obtain a closed-form local maximum likelihood estimator, a constant-memory multi-period estimator, and an asymptotic Cram\'er-Rao efficiency result. The Fisher information is governed by a combined score kernel with two additive components, a policy-mean sensitivity and a reference-path sensitivity, and a multiplicative drift sensitivity whose leading geometric contribution is a contraction of the current Hessian with the Jacobi field of the HJ characteristic flow. All three sensitivities are induced by that Jacobi field, yielding a computable link between current-field geometry and destination identifiability. We further extend the framework to a sweep-dependent observation model and formulate an active sweep-design problem for the blue team. Numerical experiments in vortex and channel-shear currents validate the Cram\'er-Rao prediction and illustrate how current geometry determines which destinations can be reliably inferred.
\end{abstract}

\end{frontmatter}

\section{Introduction}
\label{sec:intro}

Unmanned underwater vehicles (UUVs) are increasingly deployed in contested environments. Consider a blue team UUV conducting harbor operations that detects a red team UUV nearby: the red team's destination is hidden, and the blue team must infer it from observed motion alone. This is the problem of \emph{adversarial intent inference}, and it informs any subsequent response the blue team may take.

\smallskip
\noindent\textbf{An illustrative operational scenario.} A coastal surveillance platform detects a foreign UUV crossing into territorial waters. Prior intelligence narrows the candidate destinations to a small set of strategically meaningful sites (a commercial port, a submarine-cable landing, a research facility). The platform observes a noisy realization of the UUV's trajectory over a finite window and must determine which destination is most consistent with that motion, in time to support downstream tasking. The present paper characterizes the information limit of this inference problem in terms of the current-field geometry and the red team's HJ-optimal navigation. The theory is developed for a one-dimensional destination family with continuous-parameter Fisher geometry; discrete candidate sets, as in the scenario above, reduce to pairwise likelihood ratios within the same framework.

\smallskip
This paper develops a continuous-time likelihood theory for the problem in current-driven dynamics and establishes information limits set by the geometry of the Hamilton-Jacobi (HJ) travel-time field rather than by the estimator. Two features shape the underwater setting. First, UUV trajectories reflect both intent and ocean current forcing, so the same observed motion can arise from different destinations in different current fields. Second, a rational red team does not move in straight lines: it follows the HJ time-optimal path, a curved trajectory determined by the global structure of the travel-time function. This path is the characteristic of a first-order HJ equation, and its sensitivity to the destination is governed by the Jacobi field of the characteristic flow, with degeneracy at conjugate points and the cut locus.

Our main result is that what the blue team can infer is controlled by the combined score kernel
\[
  \beta(t)=\tilde g(t)+h(t)\delta X_t,\qquad
  \tilde g(t):=g(t)+\gamma(t),
\]
built from three components: an additive policy-mean channel $g$ reflecting how the HJ reference enters the red team's linear-quadratic (LQ) feedback, an additive reference-path channel $\gamma$ arising from the mismatch between the observed motion and the nominal HJ reference, and a multiplicative drift-sensitivity channel $h$ whose leading geometric contribution is a current-Hessian/Jacobi contraction. Statistically the two additive channels enter through their sum $\tilde g$: within the affine Gaussian model, nearby candidate destinations are locally identifiable if and only if $\tilde g\not\equiv0$ or $h\not\equiv0$ on $[0,T]$. The current field therefore shapes both the rational red team's trajectory and, through the LQ sensitivities and the observation-model mismatch, the information available from observing it.

To get this result, we adapt the Stackelberg-type inference framework of \cite{hu2024strategic} to HJ-constrained UUV dynamics, replacing the prescribed smooth reference of that work with the HJ-optimal path and deriving a time-varying linearized model whose coefficients depend on the current gradient along the optimal characteristic. Under a local affine approximation of the destination-to-path map, the model yields a closed-form maximum likelihood estimator with its observed-information variance, an asymptotic Cram\'er-Rao efficiency result, and a sparse multi-period estimator with constant storage. The framework then extends to a sweep-dependent observation model in which the blue team's detection strength decays with its distance to the observed vehicle: the achievable Fisher information becomes a functional of the sweep pattern, and the resulting weighted-coverage optimization is a continuous-time active experiment-design problem whose optimum is bounded by the direct-observation Fisher information. Throughout, the affine approximation supplies the closed-form statistical formulas, while the geometric interpretation of the kernels uses the full nonlinear HJ structure through the destination-to-path map.

\smallskip
\noindent\textbf{Related literature.} Path planning for autonomous underwater vehicles in current fields is well studied \cite{garau2005path,petres2007path,lekkas2014guidance}; HJ methods are developed in \cite{brandman2023globally}, and reinforcement-learning guidance extending that framework appears in \cite{greeley2023reinforcement}. Neither addresses adversarial inference. The Stackelberg-type framework of \cite{hu2024strategic} is the closest predecessor: it uses a prescribed exogenous reference, constant dynamics coefficients, and a scalar parameter with prescribed sensitivity. Replacing these with their HJ counterparts introduces destination-dependent dynamics, characteristic-based sensitivity geometry, and singularities at the cut locus, none of which appear in the prescribed-reference setting. Active inverse learning in Stackelberg trajectory games \cite{ward2025active,ward2023active} treats finite hypothesis sets in current-free environments; the continuous-parameter Fisher geometry and cut-locus singularities studied here do not arise there.

Strategic information design in adversarial systems, where agents shape what can be inferred from their actions \cite{bergemann2019information,hu2024strategic,ralston2026information,kim2026deception,zhou2025adversarial,zhou2025integrating}, provides the broader context for our work. Our problem is also related to target tracking \cite{bar2004estimation}, stochastic system identification \cite{ljung1999system,astrom1971system}, dual control \cite{feldbaum1960dual,mesbah2018stochastic}, optimal experimental design \cite{pukelsheim2006optimal}, secure state estimation \cite{shoukry2015secure,pajic2017attack}, and travel-time tomography from mean-field game dynamics \cite{xu2026traveltime}. The distinction is that target tracking estimates a moving state without a rationality constraint, whereas here the hidden variable is a fixed destination parameter that shapes the entire trajectory through the HJ equation. Girsanov-based parameter estimation for linear SDEs \cite{kutoyants2013statistical,liptser1977statistics} provides the statistical foundation; the novelty is the connection to HJ path geometry. Thus our framework gives a tractable continuous-time setting in which inference limits are computable from local score kernels determined by the HJ characteristic flow.

\smallskip
\noindent\textbf{Organization.} Section~\ref{sec:problem} formulates the UUV encounter and HJ reference dynamics. Section~\ref{sec:follower} derives the red team's entropy-regularized local policy. Section~\ref{sec:inference} develops the observable-deviation likelihood, the MLE, and the sparse estimator. Section~\ref{sec:identifiability} characterizes the Fisher information geometry and local identifiability. Section~\ref{sec:numerics} presents numerical experiments in vortex and channel-shear currents. Section~\ref{sec:sweep_design} extends the model to sweep-dependent observations and active sensing. Section~\ref{sec:conclusion} concludes.

\FloatBarrier
\section{Problem Formulation}
\label{sec:problem}

This section formulates the adversarial encounter and the red team's HJ-optimal navigation model. Sections~\ref{sec:follower}--\ref{sec:numerics} then develop and validate the direct-observation inference pipeline, while Section~\ref{sec:sweep_design} extends the framework to distance-limited observation, where the blue team also designs its sweep pattern.

\noindent\textbf{Modeling pipeline and main insight.} We study the blue team's inference problem through the pipeline in Figure~\ref{fig:pipeline}: the red team plans an HJ time-optimal path toward its hidden destination, generates a noisy trajectory $X^F_{[0,T]}$ around that path, and the blue team compares the observed trajectory with the combined score kernel $\beta(t)$, determined by the HJ path geometry, to compute the estimator $\Mhat$; the Fisher information defined by $\beta(t)$ sets the fundamental estimation limit.

The main insight is: \emph{within the affine Gaussian model, local identifiability holds precisely when $\tilde g:=g+\gamma\not\equiv0$ or $h\not\equiv0$ on $[0,T]$} (Theorem~\ref{thm:identifiability}); all three kernels $g,\gamma,h$ are induced by the destination sensitivity of the HJ characteristic flow.

\begin{figure}[!t]
\centering
\begin{tikzpicture}[node distance=0.3cm and 0.25cm]
  \node[flowbox, text width=2.9cm, fill=red!8, draw=red!50!black] (dest)
    {\textbf{Destination}\\[1pt]\small $\mathbf x^F_T=\Gamma(M)$, hidden from blue};
  \node[flowbox, text width=2.9cm, fill=teal!10, draw=teal!50!black,
        right=0.5cm of dest] (cur)
    {\textbf{Current field}\\[1pt]\small $\vc$, known to both};
  \coordinate (mid) at ($(dest.south)!0.5!(cur.south)$);
  \node[flowbox, text width=3.6cm, fill=red!8, draw=red!50!black,
        below=0.55cm of mid] (hj)
    {\textbf{Red trajectory}\\[1pt]
     \small HJ path with noise; observed $\mathbf X^F_{[0,T]}$};
  \node[flowbox, text width=3.6cm, fill=blue!8, draw=blue!50!black,
        below=0.55cm of hj] (inf)
    {\textbf{Blue team inference}\\[1pt]
     \small $\beta(t)=\gtil(t)+h(t)\delta X_t$;\; $\Mhat$, $\FI(M)$};
  \draw[flowarrow] (dest) -- (hj);
  \draw[flowarrow] (cur) -- (hj);
  \draw[flowarrow] (hj) -- node[right, font=\scriptsize]{observe} (inf);
\end{tikzpicture}
\caption{Modeling pipeline. The hidden destination and current field determine the red team's HJ-optimal reference path and induce the stochastic observed trajectory. The blue team compares the observation with its nominal model through the combined score kernel $\beta$ to compute $\Mhat$, while $\FI(M)$ gives the fundamental information limit.}
\label{fig:pipeline}
\end{figure}
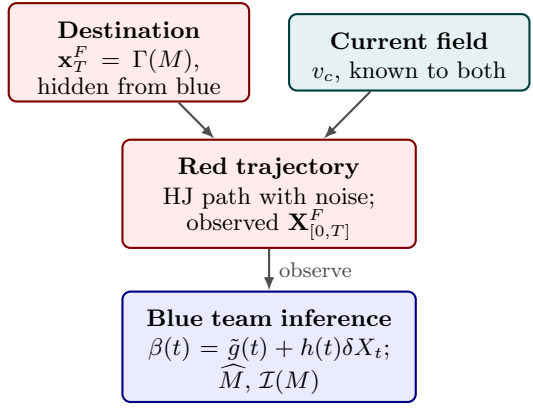

\subsection{UUV Transport Model and Current Field}
\label{subsec:transport}

Two UUVs operate in a bounded domain $\Omega \subset \R^2$ with ocean current field $\vc(\mathbf x) \in \R^2$, where $\mathbf x \in \R^2$ denotes position. For clarity we work in two spatial dimensions; the extension to three dimensions follows \cite{brandman2023globally}. The \emph{blue team} vehicle has position $\mathbf X^L_t$ and the \emph{red team} vehicle has position $\mathbf X^F_t$.

Vehicle motion follows the kinematic model of \cite{brandman2023globally} (see also \cite{fossen2021handbook}): each vehicle's absolute velocity equals its velocity relative to water plus the ocean current. The speed relative to water is bounded by $\smax>0$. We assume 
\begin{equation} \label{eq:current_bound} 
|\vc(\mathbf x)| < \smax \quad \forall \mathbf x \in \Omega, 
\end{equation} 
ensuring that minimum travel times are finite and well-defined for all destinations in $\Omega$. This holds for propelled UUVs in typical coastal currents ($\smax\sim$~1--2~m/s, $|\vc|\sim$~0.1--0.5~m/s), the setting we target, but fails for underactuated platforms such as gliders, where parts of $\Omega$ become unreachable and the HJ equation requires careful treatment at the reachability boundary \cite{lolla2014time}.

\subsection{UUV Dynamics and Red Team Reference}
\label{subsec:red_hj}

The red team is navigating toward a hidden destination $\mathbf x^F_T(M)=\Gamma(M)\in\Omega$, where $\Gamma:\mathcal M\to\Omega$ is a known smooth injection of a real parameter interval $\mathcal M\subseteq\R$ into the domain, and $M\in\mathcal M$ is the hidden \emph{destination parameter}. This covers the standard \emph{dilation family} $\Gamma(M)=M\mathbf{x}^{\mathrm{ref}}$ (with $\mathbf{x}^{\mathrm{ref}}\in\Omega$ fixed and $M$ scaling the distance along the ray) as well as general one-dimensional destination curves; the experiments in Section~\ref{sec:numerics} use both forms. The blue team observes the red team's trajectory but does not know $M$.

A rational red team vehicle, constrained by limited battery life, approximately follows the HJ time-optimal path to $\mathbf x^F_T(M)=\Gamma(M)$. The minimum arrival time $\uF(\mathbf x;M)$ satisfies the eikonal HJ equation \cite{brandman2023globally,osher1988fronts}:
\begin{equation}
  \label{eq:HJ_red}
  |\nabla \uF| - \frac{1}{\smax}\vc(\mathbf x)\cdot\nabla\uF
  = \frac{1}{\smax}, \quad \uF(\Gamma(M))=0.
\end{equation}
This is a first-order Hamilton-Jacobi equation with characteristic flow given by the optimal trajectories \eqref{eq:red_reference}.

\begin{assumption}
\label{ass:current} The current field satisfies $\vc\in C^2(\Omega)$ with $\|\nabla\vc\|_\infty<\infty$ and $\|\nabla^2\vc\|_\infty<\infty$, or the same bounds hold on a neighborhood of the relevant trajectories. This is a regularity assumption on $\vc$, not on the viscosity solution $\uF$, which is only Lipschitz in general: the $C^1$ regularity supports the characteristic description \eqref{eq:red_reference} away from singular sets of $\uF$, and the $C^2$ bound is used by the shear kernel \eqref{eq:h_explicit}.
\end{assumption}

Under Assumption~\ref{ass:current} and condition~\eqref{eq:current_bound}, equation~\eqref{eq:HJ_red} admits a unique Lipschitz-continuous viscosity solution \cite{crandall1983viscosity,bardi1997optimal,evans2002pde}, computed numerically by the fast sweeping method \cite{zhao2004fast}.

The red team's HJ-optimal reference trajectory $\xFbar(t;M)$ is the characteristic of \eqref{eq:HJ_red} through $x^F_0$:
\begin{equation}
  \label{eq:red_reference}
  \frac{d\xFvec}{dt}=-\smax\frac{\nabla\uF(\xFvec;M)}{|\nabla\uF(\xFvec;M)|}
  +\vc(\xFvec),\; \xFvec(0)=\mathbf x^F_0.
\end{equation}
The map $M\mapsto \xFvec(t;M)$ is the key analytic object of this paper. Its Jacobian $\partial_M\xFvec$ is the Jacobi field of the characteristic flow and measures how nearby optimal paths diverge as $M$ varies.

\begin{remark}
\label{rem:hj_challenges}
Replacing the prescribed smooth reference $F(t)$ of \cite{hu2024strategic} with the HJ-optimal path $\xFvec(t;M)$ introduces three analytic challenges. First, the dynamics coefficient is time-varying and $M$-dependent. Second, the viscosity solution $\uF$ is only Lipschitz continuous, so $\xFvec(t;M)$ may fail to be smooth in $M$ at values where $\mathbf x^F_0$ lies on the cut locus associated with the target $\mathbf x^F_T(M)$. Third, the HJ equation may develop nonsmoothness in $\nabla\uF$ when characteristic branches meet, producing discontinuities in $A_F(t,M)$ that have no counterpart in the smooth-reference setting.
\end{remark}

\begin{remark}
\label{rem:rational}
The HJ path is not assumed cooperative: any vehicle with a finite energy budget has an incentive to minimize travel time, and thus approximately follows the HJ-optimal path. Deviations due to, e.g., obstacle avoidance or conservative depth management are captured by the stochastic noise term in the linearized dynamics below.
\end{remark}

\begin{assumption}
\label{ass:regular_char}
For the nominal parameter $M_0\in\mathcal M$, the HJ viscosity solution $\uF(\cdot;M_0)$ is $C^2$ in a tubular neighborhood of the characteristic $\xFvec(\cdot;M_0)$, and the map $M\mapsto \xFvec(\cdot;M)$ is $C^1$ in a neighborhood of $M_0$. This assumption rules out the local cut-locus, conjugate-point, and branch-merging singularities that would destroy differentiability of the characteristic with respect to $M$.
\end{assumption}

This is a local nondegeneracy condition: the affine expansion of Assumption~\ref{ass:affine} below is well-defined precisely under it, and the Jacobi field $\Jvec(t)=\partial_M \xFvec(t;M_0)$ used throughout Section~\ref{sec:identifiability} exists as a $C^1$ function of $t$. Behavior near cut-locus or conjugate-point singularities is outside the local theory developed here.

Let $\delta \mathbf X^F_t=\mathbf X^F_t-\xFvec(t;M)$ denote the red team's deviation from its HJ reference trajectory. The perturbation dynamics are
\begin{multline}
  \label{eq:red_perturb}
  \ud\delta \mathbf X^F_t=\Bigl(A_F(t,M)\delta \mathbf X^F_t
  +\mathbf B_F\int_\R y\,\piF_t(y)\ud y\Bigr)\ud t\\
  +\sigma_F\ud \mathbf W^F_t,
\end{multline}
where $\mathbf W^F_t$ is a standard two-dimensional Brownian motion, $\sigma_F>0$, $\mathbf B_F\in\R^2$ is a fixed control-input direction, and $\piF_t$ is the red team's randomized control policy. The stochastic term $\sigma_F\ud W^F_t$ represents {process-level} randomness in the red trajectory dynamics (e.g., unmodeled sub-grid current fluctuations, control noise); we do not model additional {measurement-level} sensor noise in the blue team's observation of $\mathbf X^F_t$, which would replace the exact likelihood~\eqref{eq:loglik} with a filtering problem. The time-varying coefficient
\begin{equation}
  \label{eq:AF}
  A_F(t,M)=\nabla\vc\!\left(\xFvec(t;M)\right)
\end{equation}
is the Jacobian of the current field evaluated along the reference trajectory.

\begin{remark}
\label{rem:AF_origin}
Equation~\eqref{eq:AF} arises from a first-order Taylor expansion of $\vc(\mathbf X^F_t)$ around $\xFvec(t;M)$. A full linearization of the HJ-guided drift would also include the state derivative of the heading field $-\smax\nabla\uF/|\nabla\uF|$, which involves $\nabla^2\uF(\xFvec;M)$ through the normalization of $\nabla\uF$. In this work we model the state-dependence of the local drift through the current-shear term alone, and treat the omitted heading-linearization contribution as part of the unmodeled perturbations absorbed by the noise term $\sigma_F \ud W^F_t$. This corresponds to a shear-dominated regime, in which the heading-linearization term is small relative to the current-shear term along the reference path. Since both terms can be computed from the HJ solution and the current field, this regime can be checked numerically for a given nominal destination.
\end{remark}

\subsection{Adversarial Encounter and Information Structure}
\label{subsec:encounter}
The blue team follows a prescribed sweep pattern $\xLbar(t)$ determined by mission requirements. This pattern is first treated as part of the known observation geometry; Section~\ref{sec:sweep_design} then extends the framework to distance-dependent sensing and active sweep design.

The encounter has a sequential Stackelberg-type information structure \cite{basar1998dynamic,yong2002leader,hu2024strategic} in three steps: (1) the blue team executes its sweep $\xLvec(t)$, observable to the red team; (2) the red team navigates toward $\mathbf x^F_T(M)$, injecting randomness into its trajectory; (3) the blue team observes $\mathbf X^F_{[0,T]}$ and infers $\Mhat$. The term ``Stackelberg'' refers to this commitment-and-observation ordering inherited from \cite{hu2024strategic}; the red team's cost does not depend on the blue sweep, so the coupling is one-directional and informational rather than through payoffs: the red team's HJ-optimal behavior imprints a geometric signature on its trajectory, and the blue team reads that signature. A two-parameter extension where the red team also responds to the blue position is left for future work.

The sequential structure reflects information asymmetry: the blue team knows the red team's navigation model, dynamics, and cost structure, but neither $M$ nor the entropy weight $\lambda_F$ (Remark~\ref{rem:lambda_F} shows the latter is not needed for inference); the red team knows its own dynamics, $M$, and $\lambda_F$, and observes the blue sweep but does not model the blue team's inference objective. The red team's filtration satisfies $\sigma(\xLvec(s),0\le s\le T)\subset\mathcal F^F_0$, so the entire blue sweep is observable to the red team from time zero.

\FloatBarrier
\section{Red Team Randomized Navigation Policy}
\label{sec:follower}
For tractability of the local inference model, we work with a scalar observation coordinate. Fix a direction $e\in\R^2$ with $\|e\|=1$. Let \(\mathbf X_t^F\) and \(\xFvec(t;M)\) denote the two-dimensional state and HJ reference path. With slight abuse of notation, we write
\begin{gather*}
  X_t^F := e^\top \mathbf X_t^F,\qquad
  \xFbar(t;M):=e^\top \xFvec(t;M),\\
  \delta X_t^F:=X_t^F-\xFbar(t;M)
\end{gather*}
for the scalar projections of the state, reference, and deviation, and $B_F:=e^\top\mathbf B_F$; since $\|e\|=1$, the projected noise $e^\top\mathbf W^F_t$ is again a standard one-dimensional Brownian motion, still denoted $W^F_t$. The scalar projection of the current-shear coefficient along $e$ is
\[
  A_F(t,M):=e^\top\nabla\vc(\xFvec(t;M))\,e.
\]

\begin{assumption}
\label{ass:scalar_model}
The local deviation dynamics close, up to negligible off-axis coupling, at the scalar level with drift coefficient $A_F(t,M)$. More precisely, we neglect the coupling $e^\top\nabla\vc(\xFvec)\,e^\perp$ along the reference path, where $e^\perp$ is the orthogonal unit direction.
\end{assumption}

Assumption~\ref{ass:scalar_model} will be used for Sections~\ref{sec:follower}--\ref{sec:sweep_design}; the full vector model leads to the same structure with a matrix Riccati equation and a matrix Fisher information. Geometric quantities evaluated along the path, such as $\nabla\vc$, $\nabla^2\vc$, and the sweep distance in Section~\ref{sec:sweep_design}, always refer to the two-dimensional reference $\xFvec(t;M)$. 

We model the red team's local behavior around the HJ reference by an entropy-regularized LQ problem \cite{sutton2018reinforcement,schulman2017proximal} for the deviation dynamics. In addition to control effort and entropy regularization, the running cost contains a mild quadratic regularization in the projected state. After writing $X_t^F=\xFbar(t;M)+\delta X_t^F$, this regularization expands to
\[
  \tfrac{Q_F}{2}(\xFbar)^2
  + Q_F\,\xFbar(t;M)\,\delta X_t^F
  + \tfrac{Q_F}{2}(\delta X_t^F)^2;
\]
the constant in $\delta X_t^F$ does not affect optimization, the deviation-quadratic term contributes to the Riccati equation as usual, and the cross term $Q_F\,\xFbar(t;M)\,\delta X_t^F$ is the channel through which the HJ reference enters the policy mean. The location of the origin is immaterial: shifting it produces the same cross-term structure up to a deterministic constant.

The red team minimizes
\begin{multline}
  \label{eq:follower_obj}
  J^F(\piF;M):=\E\Bigg[\int_0^T\Bigg(
    \frac{Q_F}{2}\bigl(X^F_t\bigr)^2
    +\frac{R_F}{2}\int_\R y^2\piF_t(y)\ud y\\
    +\lambda_F\int_\R \piF_t(y)\log\piF_t(y)\ud y
  \Bigg)\ud t\Bigg],
\end{multline}
with $Q_F,R_F,\lambda_F>0$.

\begin{remark}[Interpretation of the regularization]
\label{rem:deviation_sources}
The formulation~\eqref{eq:follower_obj} describes a red team that {approximately} follows its HJ-optimal trajectory rather than realizing it exactly. The entropy regularization $\lambda_F$ and the control penalty $R_F$ together act as a tractable proxy for several operationally realistic sources of deviation that we do not model individually: (i) state and current uncertainty on the red team's side; (ii) randomized or evasive behavior (a maximum-entropy stand-in for non-deterministic decision-making \cite{jaynes1957information}); (iii) unmodeled operational constraints (depth profile, fuel, comms scheduling) that perturb the time-optimal path. The framework does not require identifying the dominant physical source of deviation: the inference results of Sections~\ref{sec:inference}--\ref{sec:identifiability} depend only on the effective parameters \(\sigma_F,Q_F,R_F\), not on whether these parameters arise from current uncertainty, control noise, randomized behavior, or operational constraints.
\end{remark}

\begin{proposition}[Red team optimal policy]
\label{prop:follower_opt}
Under Assumptions~\ref{ass:current} and~\ref{ass:scalar_model}, the red team's optimal randomized policy for \eqref{eq:follower_obj}--\eqref{eq:red_perturb} is Gaussian:
\begin{equation}
  \label{eq:follower_policy}
  \pi^{F,*}(\cdot\mid t,\delta x)=\mathcal N\!\left(
  -\frac{B_F}{R_F}\bigl(2a_t\delta x+b_t\bigr),\frac{\lambda_F}{R_F}
  \right),
\end{equation}
where $a_t$ satisfies
\begin{equation}
  \label{eq:ode_a}
  \dot a_t=\frac{2B_F^2}{R_F}a_t^2-2A_F(t,M)a_t-\frac{Q_F}{2},
  \quad a_T=0,
\end{equation}
and $b_t$ satisfies
\begin{equation}
  \label{eq:ode_b}
  \dot b_t=\frac{2B_F^2}{R_F}a_tb_t-A_F(t,M)b_t-Q_F\xFbar(t;M),
  \quad b_T=0.
\end{equation}
\end{proposition}

\begin{proof}
Expand $\tfrac{Q_F}{2}(\xFbar+\delta x)^2$ into a constant, a linear cross term $Q_F\xFbar\delta x$, and the quadratic $\tfrac{Q_F}{2}(\delta x)^2$. The HJB ansatz $V(t,\delta x)=a_t(\delta x)^2+b_t\delta x+c_t$, together with the standard Gaussian-policy minimization, yields \eqref{eq:follower_policy}. Matching coefficients of $\delta x^2$, $\delta x$, and the constant term gives \eqref{eq:ode_a}--\eqref{eq:ode_b} and an ODE for $c_t$, which does not affect the policy.
\end{proof}

\begin{remark}
\label{rem:siap_connection}
The forcing $Q_F\xFbar(t;M)$ in \eqref{eq:ode_b} makes $b_t$ nonzero and $M$-dependent: it is the channel through which the HJ reference enters the policy mean. Combined with the $M$-dependence of $A_F$ and the sensitivity of the reference trajectory itself to $M$, this yields three inference channels: an additive policy-mean channel through $\partial_M b_t$, a multiplicative drift channel through $\partial_M A_F$ (whose leading geometric content is the current Hessian contracted with the Jacobi field), and an additive reference-path channel arising from the observation-model mismatch; see Section~\ref{subsec:girsanov}. All three trace back to the Jacobi field $\partial_M\xFbar$ of the HJ characteristic flow (Section~\ref{subsec:geometric_h}). The framework of \cite{hu2024strategic} contains only an additive channel with forcing from the leader trajectory rather than the HJ reference; the multiplicative channel, the reference-path channel, and their characteristic-flow interpretation are new to the present work.
\end{remark}

\begin{proposition}[Well-posedness]
\label{thm:follower_wps}
Under Assumptions~\ref{ass:current} and~\ref{ass:scalar_model}, the Riccati ODE \eqref{eq:ode_a} and the forced linear ODE \eqref{eq:ode_b} each admit a unique solution on $[0,T]$ for all $T>0$.
\end{proposition}

\begin{proof}
Since $A_F(t,M)$ is bounded on $[0,T]$ and \eqref{eq:ode_a} has the Riccati sign of a finite-horizon LQ problem with $a_T=0$, the standard comparison principle rules out finite-time blow-up, giving a unique global solution. Given $a_t$, \eqref{eq:ode_b} is linear with bounded coefficients and $b_T=0$, hence admits a unique solution on [0,T].
\end{proof}

Under policy \eqref{eq:follower_policy}, the optimal deviation process $\delta X_t^{F,*}=X_t^{F,*}-\xFbar(t;M)$ satisfies
\begin{equation}
  \label{eq:red_state}
  \ud\delta X_t^{F,*}=\bigl[f(t,M)\delta X_t^{F,*}+d(t,M)\bigr]\ud t
  +\sigma_F\ud W_t^F,
\end{equation}
with
\begin{equation}
  \label{eq:fd_def}
  \begin{aligned}
  f(t,M)&:=A_F(t,M)-\frac{2B_F^2}{R_F}a_t(M),\\
  d(t,M)&:=-\frac{B_F^2}{R_F}b_t(M).
  \end{aligned}
\end{equation}
Thus $M$ enters the observed dynamics through both $f$ (drift coefficient along the reference, combining the current gradient $A_F$ and the Riccati response) and $d$ (policy-mean response to the reference), in addition to the reference trajectory itself.

\begin{remark}
\label{rem:observable_deviation}
The deviation $\delta X_t^{F,*}=X_t^{F,*}-\xFbar(t;M)$ of~\eqref{eq:red_state} is centered at the \emph{true} $M$ and is not directly observable to the blue team. Section~\ref{sec:inference} works with an observable deviation centered at the {nominal} $M_0$; at $M=M_0$ the two coincide.
\end{remark}

\begin{remark}
\label{rem:lambda_F}
The entropy weight $\lambda_F$ governs the variance of the randomized policy but does not appear in the optimal state dynamics \eqref{eq:red_state}. In the continuous-time formulation, the controlled drift depends only on the policy mean. Consequently the MLE and Fisher information are independent of $\lambda_F$: in an Euler discretization with step $\Delta t$, policy noise enters the one-step variance at order $(\Delta t)^2$ against $\sigma_F^2\Delta t$ for the diffusion, so its contribution vanishes as $\Delta t\to0$.
\end{remark}

\section{Intent Inference and Sparse History Tracking}
\label{sec:inference}

Section~\ref{sec:follower} derived the red team's closed-loop local dynamics. This section recenters the model at the nominal parameter $M_0$ to obtain an observable likelihood, the closed-form MLE, and a sparse multi-period estimator.

\subsection{Observable Deviation and Log-Likelihood}
\label{subsec:girsanov}

The blue team observes the projected red-team position $X_t^F$ on $[0,T]$. Since the true parameter $M$ is unknown, we work with the \emph{observable} deviation
\begin{equation}
  \label{eq:deltaX_obs_def}
  \delta X_t:=X_t^F-\xFbar(t;M_0),
\end{equation}
centered at the known nominal $M_0$. At $M=M_0$, $\delta X_t$ coincides with the ideal $\delta X_t^{F,*}$ of Remark~\ref{rem:observable_deviation}.

To obtain a closed-form estimator we localize the model around $M_0$ by combining the affine approximation of the LQ coefficients with a first-order expansion of the HJ-induced reference in $M$.

\begin{assumption}[Affine-in-parameter approximation]
\label{ass:affine}
Near the nominal $M_0$, the coefficients $f(t,M),d(t,M)$ in~\eqref{eq:fd_def} and the HJ reference $\xFbar(t;M)$ are $C^1$ in $M$ and admit first-order expansions
\begin{align}
  f(t,M)&\approx f_0(t)+(M-M_0)h(t),\label{eq:f_affine}\\
  d(t,M)&\approx d_0(t)+(M-M_0)g(t),\label{eq:d_affine}\\
  \xFbar(t;M)&\approx\xFbar(t;M_0)+(M-M_0)J(t),\label{eq:x_affine}
\end{align}
where
\begin{align}
  h(t)&:=\partial_Mf(t,M)\big|_{M_0},\label{eq:h_def}\\
  g(t)&:=\partial_Md(t,M)\big|_{M_0}
       =-\frac{B_F^2}{R_F}\partial_Mb_t\big|_{M_0},\label{eq:g_def}\\
  J(t)&:=\partial_M\xFbar(t;M_0)\label{eq:J_def}
\end{align}
are the standard local sensitivities. This is consistent with the regular-characteristic regime of Assumption~\ref{ass:regular_char} and the smooth dependence of the Riccati solutions on their coefficients. Note that $J(t)$ is the (projected) Jacobi field of the HJ characteristic flow. After reparameterizing $M\leftarrow M-M_0$, the centered expansions are $f\approx f_0+Mh$, $d\approx d_0+Mg$, and $\xFbar\approx\xFbar(t;M_0)+MJ(t)$.
\end{assumption}

In the rest of this section, $M$ denotes the centered local parameter $M-M_0$.

The dynamics of $X_t^F$ under the true parameter (centered) $M$ are
\begin{multline}
  \label{eq:XF_dyn}
  \ud X_t^F=\bigl[\xFdot(t;M)+f(t,M)(X_t^F-\xFbar(t;M))\\
  +d(t,M)\bigr]\ud t+\sigma_F\ud W_t^F.
\end{multline}
Subtracting $\xFdot(t;M_0)\ud t$ and substituting $X_t^F=\xFbar(t;M_0) +\delta X_t$, then linearizing in $M$ using~\eqref{eq:f_affine}--\eqref{eq:x_affine}, gives, to first order in $M$,
\begin{multline}
  \label{eq:deltaX_affine}
  \ud\delta X_t=\bigl[f_0(t)\delta X_t+d_0(t)\\
  +M\bigl(g(t)+\gamma(t)+h(t)\delta X_t\bigr)\bigr]\ud t
  +\sigma_F\ud W_t^F,
\end{multline}
where the new {reference-path sensitivity kernel} is
\begin{equation}
  \label{eq:gamma_def}
  \gamma(t):=\dot J(t)-f_0(t)J(t).
\end{equation}
Compared with the deviation dynamics~\eqref{eq:red_state} at the true $M$, the observable formulation~\eqref{eq:deltaX_affine} exhibits an additional deterministic drift $M\gamma(t)$ arising from the mismatch between the observed motion and the nominal reference: when the true destination differs from $M_0$, the reference we subtract off ceases to align with the underlying HJ path, and this misalignment is a third source of information about $M$.

Define the nominal drift, the combined additive kernel $\gtil:=g+\gamma$, and the combined score kernel
\begin{equation}
  \label{eq:alpha_beta_def}
  \alpha(t):=f_0(t)\delta X_t+d_0(t),\qquad
  \beta(t):=\gtil(t)+h(t)\delta X_t.
\end{equation}
Applying Girsanov's theorem \cite{karatzas1998brownian,liptser1977statistics} on $[0,T]$ to the observable process $\delta X_t$ yields
\begin{multline}
  \label{eq:loglik}
  \ell(\delta X;M)=\frac{M}{\sigma_F^2}\int_0^T\beta(t)
  \bigl[\ud\delta X_t-\alpha(t)\ud t\bigr]\\
  -\frac{M^2}{2\sigma_F^2}\int_0^T\beta(t)^2\ud t.
\end{multline}
Both $\alpha$ and $\beta$ are functions of the observed path and the known quantities $f_0,d_0,g,\gamma,h$; the likelihood \eqref{eq:loglik} is therefore computable without knowing the true $M$.

\subsection{Maximum Likelihood Estimator}
\label{subsec:MLE}

Define the observed precision
\begin{equation}
  \label{eq:G_obs}
  G_{\mathrm{obs}}:=\int_0^T\beta(t)^2\ud t
  =\int_0^T\bigl[\gtil(t)+h(t)\delta X_t\bigr]^2\ud t.
\end{equation}

\begin{proposition}
\label{prop:MLE}
Under Assumptions~\ref{ass:current}, \ref{ass:regular_char},
and~\ref{ass:affine}, when $G_{\mathrm{obs}}>0$ the MLE is
\begin{equation}
  \label{eq:MLE}
  \Mhat=\frac{\displaystyle\int_0^T\beta(t)
  \bigl[\ud\delta X_t-\alpha(t)\ud t\bigr]}{G_{\mathrm{obs}}}.
\end{equation}
\end{proposition}

\begin{proof}
The score equation $\partial\ell/\partial M=0$ applied to \eqref{eq:loglik} gives \eqref{eq:MLE}. 
\end{proof}

\begin{proposition}
\label{prop:MLE_error}
Within~\eqref{eq:deltaX_affine},
\begin{equation}
  \label{eq:Mhat_error}
  \Mhat-M=\frac{\sigma_F\int_0^T\beta(t)\ud W_t^F}{G_{\mathrm{obs}}}.
\end{equation}
The associated observed-information variance scale for the realized path is $\sigma_F^2/G_{\mathrm{obs}}$.
\end{proposition}

\begin{proof}
Substituting the SDE drift into the numerator of \eqref{eq:MLE} cancels the nominal drift against $\alpha\ud t$ and the $M\beta$ contribution against $MG_{\mathrm{obs}}$, leaving the stochastic integral in \eqref{eq:Mhat_error}. Its quadratic variation is $\sigma_F^2G_{\mathrm{obs}}$ by It\^o isometry.
\end{proof}

\subsection{Sparse Multi-Period Estimator}
\label{subsec:sparse}

Over $N$ independent observation periods of length $T$, period $i$ produces an MLE $\Mhat^{(i)}$ with observed precision $G_{\mathrm{obs}}^{(i)}$. Define
\begin{equation}
  \label{eq:PN}
  P_N:=\sum_{i=1}^NG_{\mathrm{obs}}^{(i)}.
\end{equation}

\begin{proposition}
\label{prop:sparse}
The estimator
\[
  \Mbar_N:=P_N^{-1}\sum_{i=1}^NG_{\mathrm{obs}}^{(i)}\Mhat^{(i)}
\]
admits the $O(1)$-storage recursion
\begin{align}
  P_{N+1}&=P_N+G_{\mathrm{obs}}^{(N+1)},\label{eq:update_P}\\
  \Mbar_{N+1}&=P_{N+1}^{-1}\bigl(P_N\Mbar_N
  +G_{\mathrm{obs}}^{(N+1)}\Mhat^{(N+1)}\bigr),\label{eq:update_M}
\end{align}
with observed-information variance scale $\sigma_F^2/P_N$.
\end{proposition}

\begin{proof}
Per-period observed-information variance scales $\sigma_F^2/G_{\mathrm{obs}}^{(i)}$ combine under precision weighting to $\sigma_F^2/P_N$; the recursion follows from expanding $P_{N+1}\Mbar_{N+1}$.
\end{proof}

\begin{proposition}
\label{prop:asymptotic}
Under Assumptions~\ref{ass:current}, \ref{ass:regular_char},
and~\ref{ass:affine}, as $N\to\infty$,
\[
  \Mbar_N\to M \;\text{a.s.},\quad
  \sqrt N(\Mbar_N-M)\xrightarrow{d}\mathcal N(0,\FI(M)^{-1}),
\]
where $\FI(M)=\E_M[G_{\mathrm{obs}}]/\sigma_F^2$ is the Fisher information \eqref{eq:FIM}. The estimator attains the Cram\'er-Rao benchmark asymptotically.
\end{proposition}

\begin{proof}
The error representation \eqref{eq:Mhat_error} gives independent zero-mean per-period scores. The strong law gives $P_N/N\to\E_M[G_{\mathrm{obs}}]$ a.s., and the martingale central limit theorem applied to the precision-weighted sum yields the Gaussian limit with variance $\FI(M)^{-1}$.
\end{proof}

Within the centered affine model~\eqref{eq:deltaX_affine}, the multi-period framework is a repeated-experiment setting: each period is an independent replication of the encounter under the same configuration (initial condition, local model, and observation window), e.g., repeated transits by vehicles of the same class in the same candidate-destination scenario. A period is not a segmentation of a single trajectory, which would violate independence across periods since the deviation process does not reset.

\section{Fisher Information Geometry}
\label{sec:identifiability}

Throughout this section $M$ denotes the centered parameter of Assumption~\ref{ass:affine}. For the affine model~\eqref{eq:deltaX_affine}, the Fisher information at centered parameter $M$ is
\begin{equation}
  \label{eq:FIM}
  \FI(M)=\frac{1}{\sigma_F^2}\int_0^T
  \E_M\!\left[\bigl(\gtil(t)+h(t)\delta X_t\bigr)^2\right]\ud t.
\end{equation}
Let
$
  m_M(t):=\E_M[\delta X_t],\;
  V_M(t):=\E_M[(\delta X_t)^2].
$ 
Then
\begin{align}
  \dot m_M(t)&=(f_0(t)+Mh(t))m_M(t)\notag\\
  &\quad+d_0(t)+M\gtil(t),\quad m_M(0)=0,\label{eq:mean_ode}\\
  \dot V_M(t)&=2(f_0(t)+Mh(t))V_M(t)\notag\\
  &\quad+2\bigl(d_0(t)+M\gtil(t)\bigr)m_M(t)+\sigma_F^2,\notag\\
  &\hspace{9.5em} V_M(0)=0.\label{eq:variance_ode}
\end{align}
Consequently,
\begin{equation}
  \label{eq:FIM_expand}
  \FI(M)=\frac{1}{\sigma_F^2}\int_0^T
  \bigl(\gtil^2+2\gtil h\,m_M+h^2V_M\bigr)\ud t.
\end{equation}

\subsection{Jacobi-field decomposition of the kernels}
\label{subsec:geometric_h}
The three terms correspond to the combined additive kernel (policy-mean plus reference-path), the cross-channel contribution, and the multiplicative drift channel. Although the cross term may be negative pointwise, the full integrand is nonnegative because it equals $\E[(\gtil(t)+h(t)\delta X_t)^2]$.

All three kernels $g,\gamma,h$ decompose through the Jacobi field of the HJ characteristic flow: the two-dimensional field is $\Jvec(t):=\partial_M\xFvec(t;M_0)$, and its projection $J(t):=\partial_M\xFbar(t;M_0)=e^\top\Jvec(t)$ is the scalar sensitivity of Assumption~\ref{ass:affine}. The reference-path kernel is directly Jacobi-driven, $\gamma(t)=\dot J(t)-f_0(t)J(t)$, while $g$ and $h$ arise through the LQ sensitivity equations below. Recall that after projection onto the observation direction $e$, the scalar drift coefficient is $A_F(t,M)=e^\top\nabla\vc\bigl(\xFvec(t;M)\bigr)e$ by \eqref{eq:AF}, so
\begin{equation}
  \label{eq:dM_AF}
  \partial_MA_F(t,M_0)
    =e^\top\bigl[D^2\vc\bigl(\xFvec(t;M_0)\bigr)[\Jvec(t)]\bigr]e,
\end{equation}
where $D^2\vc[\Jvec]$ denotes the tensor contraction of the Hessian $D^2\vc$ with the vector $\Jvec$. For notational brevity we write $\nabla^2\vc\cdot\Jvec$ for this projected contraction in the sequel. The shear kernel is
\begin{equation}
  \label{eq:h_explicit}
  h(t)=\nabla^2\vc(\xFvec(t;M_0))\cdot\Jvec(t)
  -\frac{2B_F^2}{R_F}a_M(t),
\end{equation}
where $a_M(t):=\partial_Ma_t|_{M_0}$ solves
\begin{multline}
  \label{eq:dM_a_ode}
  \dot a_M(t)=\left(\frac{4B_F^2}{R_F}a_t-2A_F\right)a_M(t)\\
  -2a_t\,\partial_MA_F(t,M_0),\qquad a_M(T)=0.
\end{multline}
The policy-mean kernel is
\begin{equation}
  \label{eq:g_kernel_explicit}
  g(t)=-\frac{B_F^2}{R_F}b_M(t),
\end{equation}
where $b_M(t):=\partial_Mb_t|_{M_0}$ satisfies
\begin{multline}
  \label{eq:dM_b_ode}
  \dot b_M(t)=\left(\frac{2B_F^2}{R_F}a_t-A_F\right)b_M(t)
  +\frac{2B_F^2}{R_F}b_ta_M(t)\\
  -b_t\,\partial_MA_F(t,M_0)-Q_FJ(t),\qquad b_M(T)=0.
\end{multline}
Equations \eqref{eq:dM_a_ode}--\eqref{eq:dM_b_ode} show how the HJ Jacobi field propagates into the LQ-induced kernels $g$ and $h$, while the reference-path kernel $\gamma=\dot J-f_0 J$ is directly determined by the Jacobi field. In particular, if $J\equiv0$, then $\partial_MA_F\equiv0$; the linearized Riccati equation for $a_M$ is homogeneous with zero terminal condition, so $a_M\equiv0$. The equation for $b_M$ is then also homogeneous, since its forcing $-Q_FJ$ vanishes, and hence $b_M\equiv0$. The reference-path kernel $\gamma=\dot J-f_0 J$ also vanishes trivially when $J\equiv0$. Thus all three kernels $g,\gamma,h$ vanish when $J\equiv0$ on the observation interval $[0,T]$, not merely at an isolated time.

\begin{remark}
\label{rem:reduction}
When $\vc(x)=Ax+b$ is linear, $A_F$ is constant in $t$ and $M$, so $\partial_MA_F\equiv0$. The linearized Riccati equation~\eqref{eq:dM_a_ode} is then homogeneous with $a_M(T)=0$, giving $a_M\equiv0$ and hence, by~\eqref{eq:h_explicit}, $h\equiv0$: the multiplicative channel is activated precisely by current curvature. Identifiability in linear currents rests entirely on the additive channels, which remain active through the Jacobi field $J$, directly in $\gamma=\dot J-f_0J$ and through the forcing $-Q_FJ$ in~\eqref{eq:dM_b_ode} for $g$.
\end{remark}

\subsection{Identifiability}
\label{subsec:identifiability}

\begin{theorem}
\label{thm:identifiability}
Assume Assumptions~\ref{ass:current}, \ref{ass:regular_char}, and~\ref{ass:affine}, and that $\sigma_F>0$, and recall $\gtil=g+\gamma$ and $\beta=\gtil+h\,\delta X_t$ from~\eqref{eq:alpha_beta_def}. Then $M_0$ is locally identifiable in the affine Gaussian model if and only if
\[
    \tilde g\not\equiv0\quad\text{or}\quad h\not\equiv0
    \quad\text{on }[0,T].
\]
Equivalently,
\[
  \FI(M_0)=\frac{1}{\sigma_F^2}\int_0^T\E[\beta(t)^2]\ud t>0.
\]
If $\tilde g\equiv0$ and $h\equiv0$, then $\FI(M_0)=0$.
\end{theorem}

\begin{proof}
If $\tilde g\equiv0$ and $h\equiv0$, then $\beta\equiv0$, so $\FI(M_0)=0$. Conversely, suppose $\tilde g\not\equiv0$ or $h\not\equiv0$. Define
\[
  S:=\{t\in(0,T]:|\tilde g(t)|+|h(t)|>0\};
\]
$S$ has positive Lebesgue measure by hypothesis. Writing $m(t):=m_{M_0}(t)$ and $V(t):=V_{M_0}(t)$,
\begin{multline*}
  \E[(\gtil(t)+h(t)\delta X_t)^2]\\
  =(\gtil(t)+h(t)m(t))^2+h(t)^2\bigl(V(t)-m(t)^2\bigr).
\end{multline*}
For $t>0$, $V(t)-m(t)^2=\Var(\delta X_t)>0$ because $\sigma_F>0$. On $S$, either $h(t)\neq0$ (and the second term is strictly positive), or $h(t)=0$ and $\tilde g(t)\neq0$ (and the first term equals $\tilde g(t)^2>0$). Hence the integrand is positive on $S$, giving $\FI(M_0)>0$ and local identifiability.
\end{proof}

The theorem is local: it characterizes distinguishability from infinitesimal perturbations within the affine model centered at $M_0$. Recomputing $\xFbar$, $J$, and the kernels at other nominal values traces out the map $M\mapsto\FI(M)$ plotted in Section~\ref{sec:numerics}.

\begin{remark}
\label{rem:operational}
The kernels $g,\gamma,h$ are determined by the current field, the HJ solution, the Jacobi field, and the LQ sensitivity equations, without observational data: a blue team can evaluate $\FI(M)$ over a candidate destination set {before deployment} to identify where the current geometry supports inference and where it does not.
\end{remark}

\section{Numerical Results}
\label{sec:numerics}
We test the framework in two current fields: a counter-clockwise harbor vortex and an east-flowing channel-shear current. The experiments numerically validate the main predictions of the framework: the combined score kernel predicts the estimator's variance through the Cram\'er-Rao benchmark, and the current geometry determines where inference quality is high or low.

Parameters: $\smax=1.5$ m/s, $\sigma_F=0.15$ m\,s$^{-1/2}$, $Q_F=10^{-3}$, $R_F=\lambda_F=B_F=1.0$, $T=800$ s, and $\Delta t=1$ s. All simulations use the scalar observation direction \(e=(1,0)^\top\). The HJ equation is solved by fast sweeping \cite{zhao2004fast} on a $101\times101$ grid with $\Delta x=20$ m. Monte Carlo estimates use 150--600 sample paths; the HJ reference and the kernels $g,\gamma,h$ are recomputed at each candidate $M$, so the full nonlinear HJ-induced parameter dependence enters through the kernels. We compare HJ-optimal references against a \emph{dead-reckoning} baseline that steers directly toward the destination at maximum speed without optimizing over the current.

\subsection{Vortex current}
\label{subsec:exp_vortex}
The counter-clockwise harbor vortex has peak speed $1.1$ m/s at $r=150$ m, $\mathbf x = (x_1, x_2)$,
\begin{equation}
  \label{eq:vortex}
  \vc(\mathbf x)=\frac{1.1\,r}{150+r}\frac{(-x_2,x_1)^\top}{\|\mathbf x\|_2},
  \quad r=\|\mathbf x\|_2.
\end{equation}
In the numerical implementation, the value at $r=0$ is defined by a small-radius regularization. Three vehicles depart from $(-500,450)$, $(-600,0)$, and $(-500,-450)$ m toward destinations $M\cdot(400,0)$ m with $M\in\{0.8,1.0,1.2\}$.

Figure~\ref{fig:vortex_traj} shows the trajectories: the HJ paths arc with the rotation while dead-reckoning paths point directly at the target, so the HJ paths sample a wider variety of current gradients, which is reflected in their information advantage below.

\begin{figure}[htbp]
  \centering
  \includegraphics[width=\linewidth]{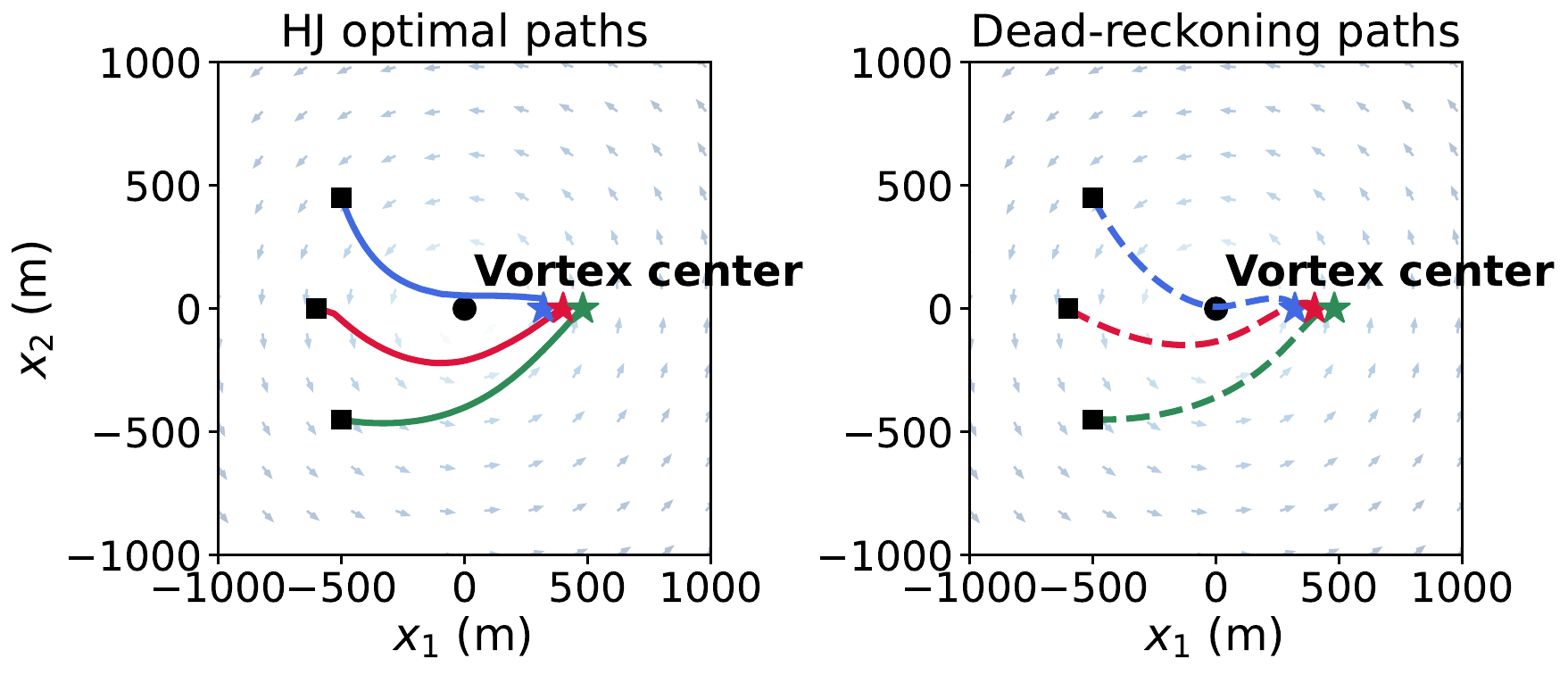}
  \caption{Vortex field \eqref{eq:vortex} with three trajectories for $M\in\{0.8,1.0,1.2\}$ (blue/red/green). Squares: starts; stars: destinations; black dot: vortex center. Left: HJ-optimal; right: dead-reckoning.}
  \label{fig:vortex_traj}
\end{figure}

\paragraph{Cram\'er-Rao validation.}
Figure~\ref{fig:cr_validation} is the central empirical result. Across $M\in[0.5,1.5]$, the simulated single-period MLE variance (markers) tracks the theoretical Cram\'er-Rao bound $\FI(M)^{-1}$ (lines) for both references, validating the three-channel Fisher information formula. The HJ advantage grows with $M$ as the dead-reckoning trajectories enter a region where the combined score kernel is smaller: the empirical variance ratio is approximately unity at $M=0.5$, $1.4\times$ at $M=1.0$, and $3.1\times$ at $M=1.5$.

\begin{figure}[htbp]
  \centering
  \includegraphics[width=0.74\linewidth]{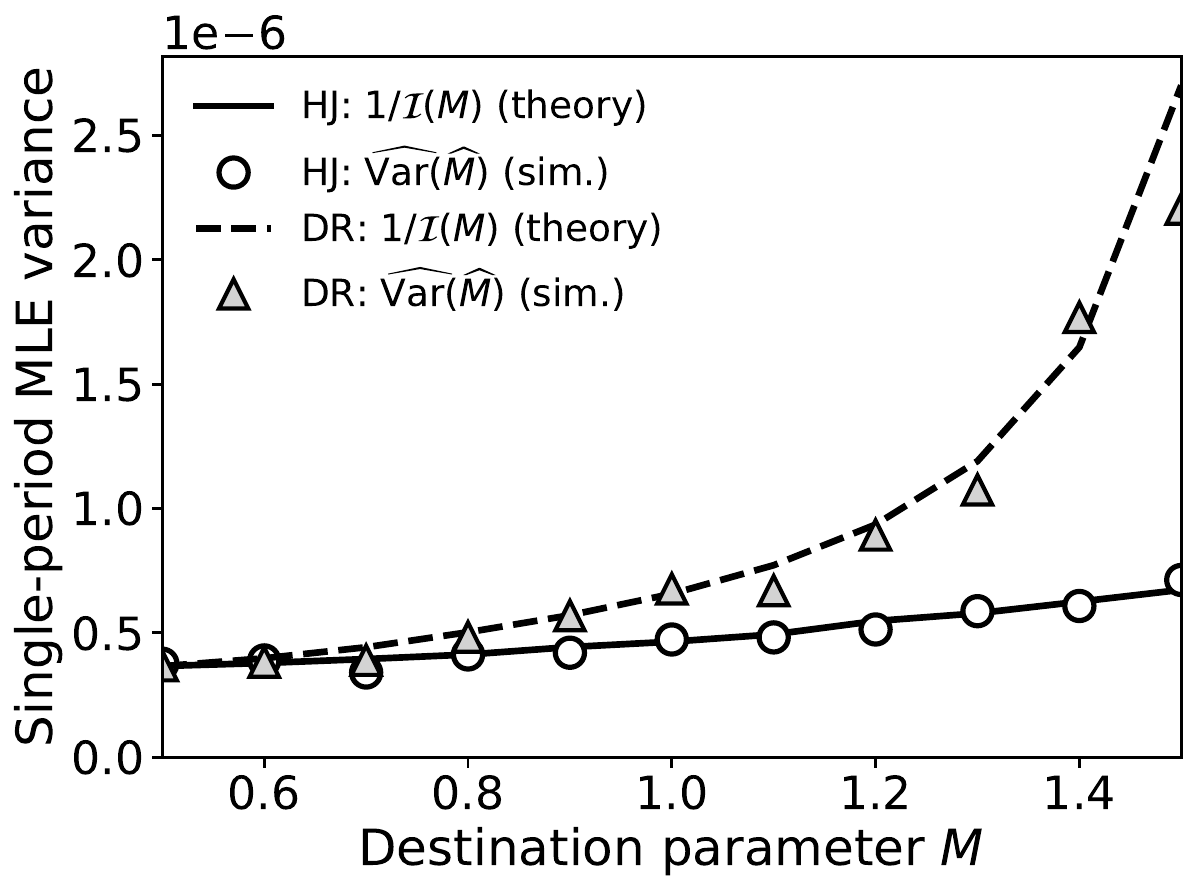}
  \caption{Cram\'er-Rao validation in the vortex field. Empirical single-period MLE variance $\widehat{\mathrm{Var}}(\Mhat)$ (markers) overlays the theoretical Cram\'er-Rao bound $\FI(M)^{-1}$ (lines) for HJ-optimal (solid/circles) and dead-reckoning (dashed/triangles) references.}
  \label{fig:cr_validation}
\end{figure}

\paragraph{Fisher information geometry.}
Figure~\ref{fig:vortex_FI} plots $\FI(M)$ on its own scale. Both curves show an overall decreasing trend in $M$, reflecting a decrease in the effective score-kernel energy accumulated along the observed path over the fixed horizon $T$; HJ lies above dead-reckoning throughout. The mean ratio over the range is approximately $2.0\times$, with the gap widening toward larger $M$ as the dead-reckoning kernel decreases relative to the HJ kernel.

\begin{figure}[htbp]
  \centering
  \includegraphics[width=0.74\linewidth]{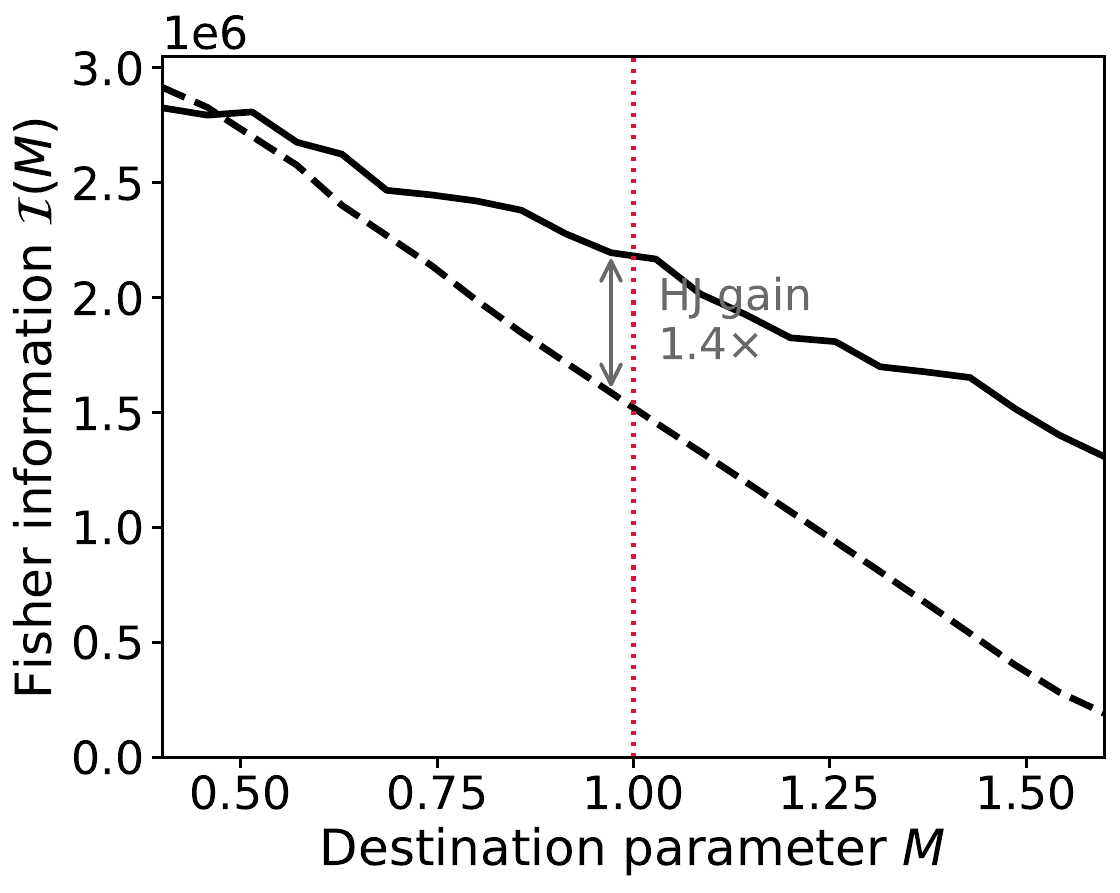}
  \caption{Fisher information $\FI(M)$ in the vortex field. HJ (solid); dead-reckoning (dashed); $M=1.0$ marked.}
  \label{fig:vortex_FI}
\end{figure}

\paragraph{Multi-period sparse estimator.}
Figure~\ref{fig:vortex_sparse} shows sample paths and variance of $\Mbar_N$ over $N=20$ periods. Both estimators converge to $M=1.0$; the variance decay matches the $1/N$ rate predicted by Proposition~\ref{prop:asymptotic}, with HJ maintaining its single-period advantage throughout.

\begin{figure}[htbp]
  \centering
  \includegraphics[width=\linewidth]{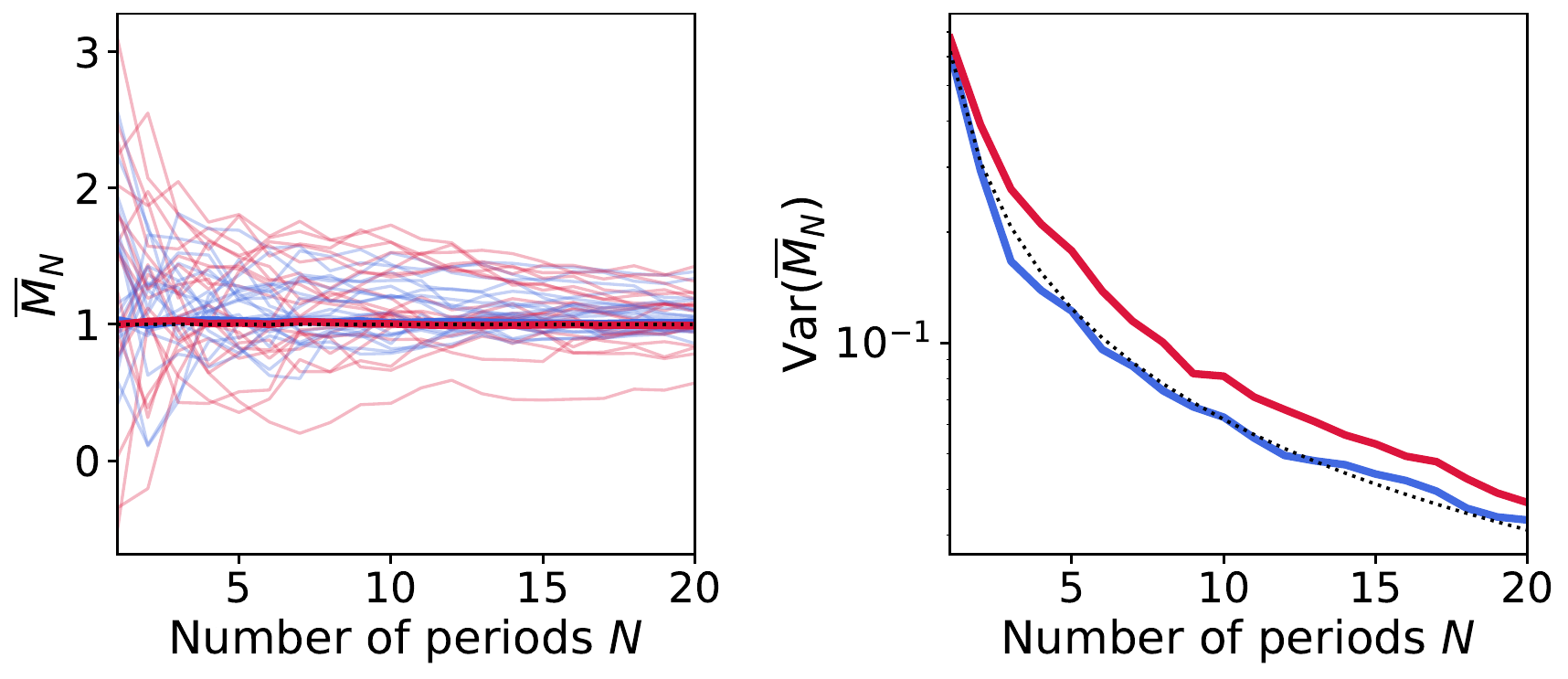}
  \caption{Sparse estimator $\Mbar_N$ at $M=1.0$ over $N=20$ periods. Left: sample paths, HJ (blue) and dead-reckoning (red); thick lines are means; dotted: true value. Right: variance on log scale with $1/N$ reference (dotted).}
  \label{fig:vortex_sparse}
\end{figure}

\subsection{Channel shear current}
\label{subsec:exp_shear}

The east-flowing channel decays with distance from the centerline:
\begin{equation}
  \label{eq:shear_current}
  \vc(x)=\bigl(v_0e^{-x_2^2/L^2},0\bigr),
  \quad v_0=1.2\;\text{m/s},\; L=250\;\text{m}.
\end{equation}
A single vehicle starts at $(-400,600)$ m. For the trajectory illustration in Figure~\ref{fig:shear_traj} we take three off-ray destinations $(400,400)$, $(500,0)$, and $(300,-300)$ m, labeled $M\in\{0.8,1.0,1.2\}$, chosen to sample the shear geometry above, along, and below the centerline. For the continuous information plot in Figure~\ref{fig:FI_comparison} the destination curve is taken as the dilation family along the centerline,
\begin{equation}
  \label{eq:shear_Gamma}
  \Gamma(M) = M\cdot(500,0)\,\text{m},
\end{equation}
so that a scalar $M$ parameterization is well defined over $M\in[0.5,1.6]$ and the Jacobi field $J(t)=\partial_M\xFbar(t;M_0)$ of Section~\ref{sec:identifiability} is well defined at each nominal $M_0$. The same $\Gamma$ is used to compute the HJ reference and the kernels $g,\gamma,h$ throughout the FI scan. Thus Figure~\ref{fig:shear_traj} is a qualitative path-geometry illustration, whereas the quantitative Fisher-information scan in Figure~\ref{fig:FI_comparison} uses the scalar family~\eqref{eq:shear_Gamma}.

Figure~\ref{fig:shear_traj} shows the trajectories. The HJ paths dip toward the centerline to exploit the eastward current, then curve to their destinations; dead-reckoning paths cross the centerline more steeply and spend less time near the high-current, high-curvature region around $x_2=0$.

\begin{figure}[htbp]
  \centering
  \includegraphics[width=\linewidth]{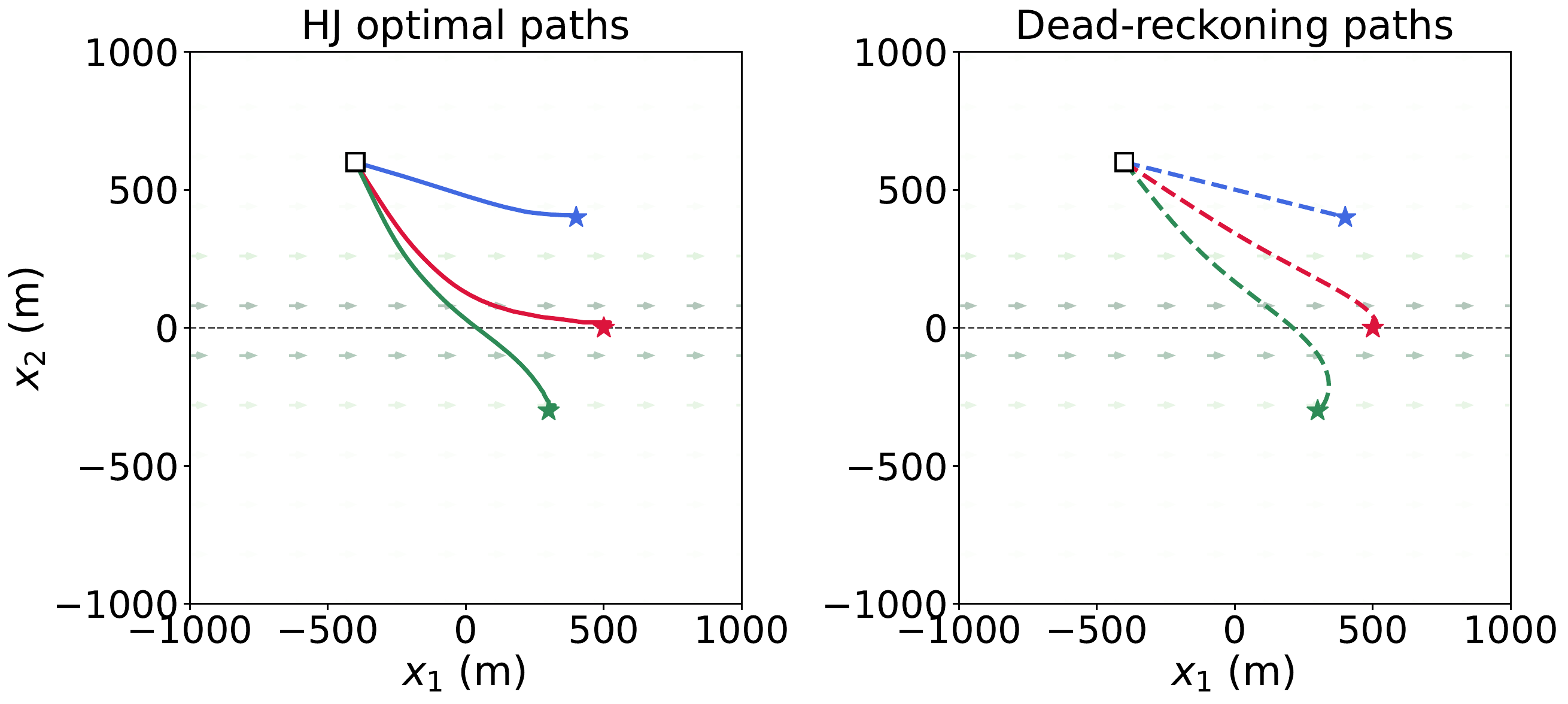}
  \caption{Channel shear field \eqref{eq:shear_current}, common start $(-400,600)$ m; three destinations spanning the centerline. Dashed line: $x_2=0$ (peak current, zero shear, peak curvature). Left: HJ; right: dead-reckoning.}
  \label{fig:shear_traj}
\end{figure}

\begin{figure}[htbp]
  \centering
  \includegraphics[width=0.75\linewidth]{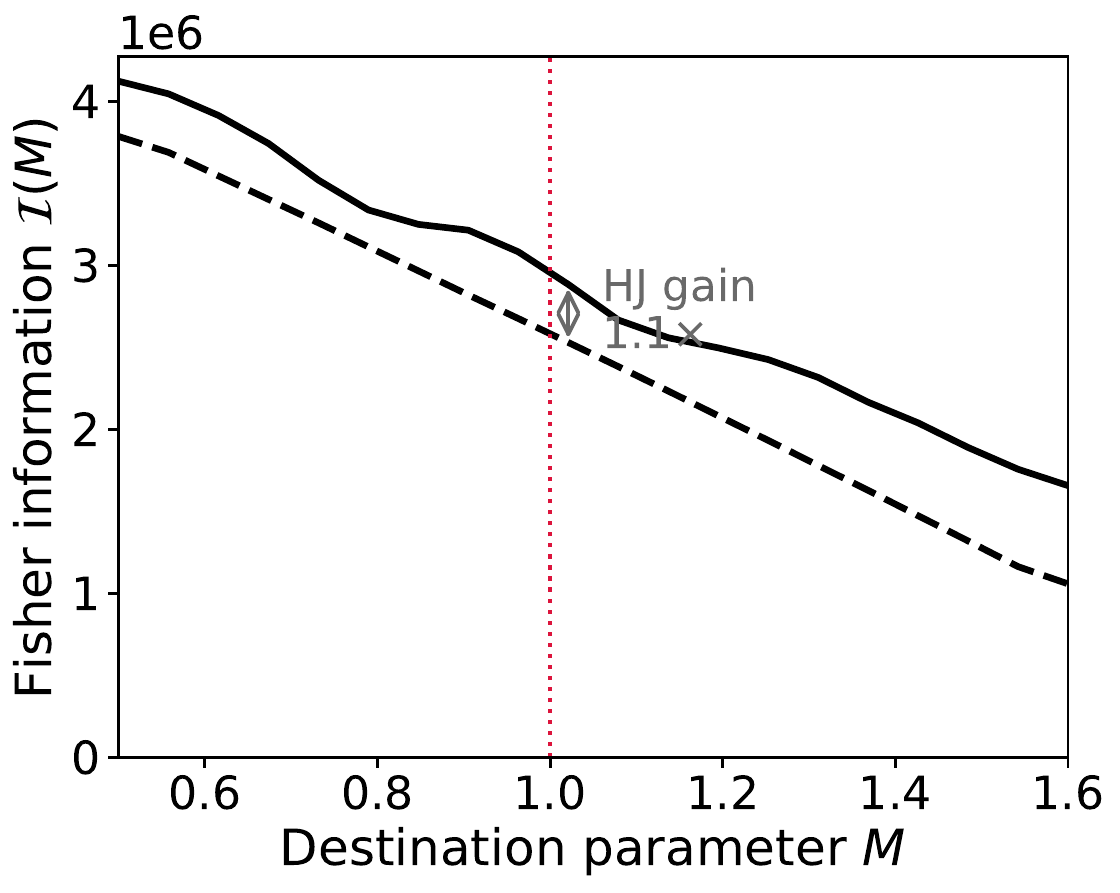}
  \caption{Fisher information $\FI(M)$ in the channel-shear field. HJ (solid); dead-reckoning (dashed); $M=1.0$ marked. HJ maintains an advantage across the tested range with the gap widening toward larger $M$.}
  \label{fig:FI_comparison}
\end{figure}

\paragraph{Path-dependence of information.}
Figure~\ref{fig:FI_comparison} plots $\FI(M)$ for both references. HJ lies above dead-reckoning throughout the tested range, with a modest gap (about $1.1\times$ at $M=1.0$) that widens toward larger $M$. The three channels of Theorem~\ref{thm:identifiability} respond differently to the shear geometry: the multiplicative drift-sensitivity channel $h$ favors trajectories that sample regions where the current-Hessian / Jacobi contraction is large, which in this example is associated with the high-curvature region near the centerline (a pattern where dead-reckoning paths can be competitive with HJ), while the additive channels $g$ and $\gamma$ favor trajectories whose destination sensitivity is largest, an advantage that HJ paths retain. In the present current field the additive channels dominate and HJ wins throughout; a scenario dominated by the $h$-channel would recover the intuition that a straight path through the high-curvature region can be more informative than a longer HJ path.

\FloatBarrier
\section{Active Sweep Design}
\label{sec:sweep_design}

Sections~\ref{sec:problem}--\ref{sec:numerics} treat the blue team's sweep $\xLbar$ as a prescribed mission pattern. In practice the blue team can choose its sweep to maximize the information it gathers about $M$. This section extends the framework to a sweep-dependent observation model and formulates the resulting design problem. The red team's HJ-optimal dynamics, the LQ sensitivity structure of Sections~\ref{sec:follower}--\ref{sec:identifiability}, and the combined-kernel Fisher information are all retained; only the blue team's observation channel changes.

\subsection{Distance-attenuated observation model}
\label{subsec:sweep_obs_model}

Let $\rho:\R_+\to[0,1]$ be a {detection function}: continuous, non-increasing, with $\rho(0)=1$ and $\lim_{s\to\infty}\rho(s)=0$. With
\begin{equation}
  r(t):=\|\xLvec(t)- \xFvec(t;M_0)\|
\end{equation}
the distance from the blue sweep to the nominal HJ path, we posit the following \emph{effective-noise observation model}: the blue team observes $\delta X_t$ through a channel whose per-unit-time noise standard deviation scales as $\sigma_F/\rho(r(t))$, so that the signal-to-noise ratio degrades with distance. Under this effective-noise model, the local likelihood is obtained from~\eqref{eq:loglik} by the multiplicative weight $\rho(r(t))^2$ under the integrals (see~\eqref{eq:loglik_sweep} below), and~\eqref{eq:FI_sweep} is the corresponding effective Fisher information. The direct-observation setting of Sections~\ref{sec:inference}--\ref{sec:identifiability} corresponds to $\rho\equiv1$. Two natural choices are
\[
  \rho_{\text{sharp}}(r)=\mathbf{1}\{r\le R\},\qquad
  \rho_{\text{smooth}}(r)=e^{-r^2/(2L^2)},
\]
modeling a sharp detection radius $R$ and a Gaussian rolloff with characteristic length $L$. The framework below requires only monotonicity of $\rho$; the existence argument of Section~\ref{subsec:sweep_problem} additionally uses continuity, which excludes $\rho_{\text{sharp}}$ but not smooth approximations of it.

\begin{remark}
\label{rem:effective_noise}
The effective-noise model is a stylized reduction of sensing degradation; physically, the red team's process noise is $\sigma_F$ regardless of the blue position. If the data are instead generated by the direct-observation model~\eqref{eq:deltaX_affine} and~\eqref{eq:loglik_sweep} is used as a weighted quasi-likelihood, the estimator~\eqref{eq:MLE_sweep} retains the error representation of Proposition~\ref{prop:MLE_error} with $\beta$ replaced by $\rho^2\beta$ in the stochastic integral, yielding the sandwich variance scale
\[
  \sigma_F^2\,\frac{\int_0^T\rho(r(t))^4\beta(t)^2\ud t}
  {\bigl(\int_0^T\rho(r(t))^2\beta(t)^2\ud t\bigr)^{2}}.
\]
Under the effective-noise model adopted in this section, the corresponding Fisher variance is $\FI_Y^{-1}$; the two scales agree, e.g., when $\rho\equiv1$, but differ in general. In this paper we adopt~\eqref{eq:sweep_design} as the effective-noise design criterion; under the direct-observation quasi-likelihood reading, a fully sandwich-optimal design would instead involve the corresponding sandwich ratio. We do not model an additional stochastic time-thinning of the observation.
\end{remark}

\subsection{Sweep-dependent effective Fisher information}
\label{subsec:sweep_FI}

Under Assumption~\ref{ass:affine} and the effective-noise observation model of Section~\ref{subsec:sweep_obs_model}, the local log-likelihood is
\begin{multline}
  \label{eq:loglik_sweep}
  \ell_{\xLbar}(M)=\frac{M}{\sigma_F^2}\!\int_0^T\!\rho(r(t))^2
  \beta(t)\bigl[\ud\delta X_t-\alpha(t)\ud t\bigr]\\
  -\frac{M^2}{2\sigma_F^2}\!\int_0^T\!\rho(r(t))^2\beta(t)^2\ud t,
\end{multline}
where $\alpha,\beta$ are as in \eqref{eq:alpha_beta_def}. The MLE
\begin{equation}
  \label{eq:MLE_sweep}
  \Mhat_{\xLbar}=\frac{\displaystyle\int_0^T\rho(r(t))^2\beta(t)
  \bigl[\ud\delta X_t-\alpha(t)\ud t\bigr]}
  {\displaystyle\int_0^T\rho(r(t))^2\beta(t)^2\ud t}
\end{equation}
generalizes~\eqref{eq:MLE}, and the corresponding Fisher information is
\begin{equation}
  \label{eq:FI_sweep}
  \FI_Y(M_0; \xLvec)=\frac{1}{\sigma_F^2}\int_0^T\rho(r(t))^2\,
  \E[\beta(t)^2]\,\ud t.
\end{equation}
Equation~\eqref{eq:FI_sweep} is the original information~\eqref{eq:FIM_expand} weighted pointwise by the sweep-dependent factor $\rho(r(t))^2$. Three properties follow at once:
\begin{enumerate}[label=(\roman*)]
\item $\FI_Y(M_0;\xLvec)\le\FI(M_0)$, with equality iff
  $\rho(r(t))=1$ a.e.
\item $\FI_Y(M_0; \xLvec)\to0$ as $\inf_{t\in[0,T]}r(t)\to\infty$.
\item $\FI_Y$ is monotone in $\rho^2$ pointwise.
\end{enumerate}
The sweep enters \eqref{eq:FI_sweep} only through the deterministic weight $\rho(r(t))^2$; the kernel $\E[\beta(t)^2]$ is a property of the red team's dynamics, precomputed once from \eqref{eq:mean_ode}--\eqref{eq:variance_ode}. The information geometry of Section~\ref{sec:identifiability} therefore plays the role of a \emph{prize density} over time.

\subsection{Sweep design problem}
\label{subsec:sweep_problem}

Let $\mathcal{X}^L$ denote the feasible set of sweep patterns. A natural choice is
\begin{multline}
  \label{eq:sweep_feasible}
  \mathcal{X}^L=\bigl\{\xLvec\in W^{1,\infty}([0,T];\Omega):\\
  \xLvec(0)=\mathbf x_0^L,\;\|\dot{\bar{\mathbf{x}}}^L\|\le v_L\bigr\}
\end{multline}
for an initial position $x_0^L\in\Omega$ and a speed bound $v_L>0$. The blue team's design problem is
\begin{multline}
  \label{eq:sweep_design}
  \xLvec_{\mathrm{opt}}=\arg\max_{\xLvec\in\mathcal{X}^L}\\
  \int_0^T \rho\bigl(\|\xLvec(t)-\xFvec(t;M_0)\|\bigr)^2
  \E[\beta(t)^2]\,\ud t.
\end{multline}
This is a continuous-time weighted-coverage problem: maximize weighted overlap with the predicted HJ path under sensor coverage $\rho^2$ and prize density $\E[\beta^2]$. Existence follows from standard compactness, provided the feasible curves remain in a compact subset of $\Omega$: bounded speed gives equicontinuity, Arzel\`a--Ascoli gives sequential compactness of $\mathcal{X}^L$ in the uniform topology, and the objective is continuous in $\xLvec$ under continuity of $\rho$.

\subsection{Tracking converse}
\label{subsec:sweep_tracking}

The next result connects the information limit of Section~\ref{sec:identifiability} (which assumes direct observation) to the achievable information under a speed-constrained sweep.

\begin{proposition}
\label{prop:tracking}
Assume Assumptions~\ref{ass:current}, \ref{ass:regular_char}, and~\ref{ass:affine}, $\sigma_F>0$, $\rho$ continuous and non-increasing on $[0,\infty)$ with $\rho(0)=1$, and $\FI(M_0)>0$.
\begin{enumerate}[label=(\roman*)]
\item If $\xFbar(\cdot;M_0)\in\mathcal{X}^L$ (for the feasible
  set~\eqref{eq:sweep_feasible}, this holds when $\mathbf x_0^L=\mathbf x_0^F$ and
  $v_L\ge\sup_{t\in[0,T]}\|\xFvdot(t;M_0)\|$), then
  $\sup_{\xLvec\in\mathcal{X}^L}\FI_Y(M_0;\xLvec)=\FI(M_0)$, attained
  at $\xLvec=\xFvec(\cdot;M_0)$.
\item If $\rho$ is strictly decreasing on $[0,\infty)$ and
  $\xLvec\in\mathcal{X}^L$ satisfies $\xLvec\ne\xFvec(\cdot;M_0)$ on a
  set $S\subset[0,T]$ such that
  $\{t\in S:\E[\beta(t)^2]>0\}$ has positive Lebesgue measure, then
  $\FI_Y(M_0;\xLvec)<\FI(M_0)$.
\end{enumerate}
\end{proposition}

\begin{proof}
For (i): if $\xLvec=\xFvec(\cdot;M_0)$, then $r(t)\equiv0$, so $\rho(r(t))^2\equiv1$, and \eqref{eq:FI_sweep} reduces to~\eqref{eq:FIM_expand}, giving $\FI_Y(M_0;\xLvec)=\FI(M_0)$. The upper bound $\FI_Y\le\FI$ holds for any $\xLvec$ by property~(i) of~\eqref{eq:FI_sweep}; the supremum is therefore attained at $\xFvec(\cdot;M_0)$. 

For (ii): the information gap is 
\begin{multline} 
  \label{eq:tracking_gap}
\Delta \FI=\FI(M_0)-\FI_Y(M_0;\xLvec)\\ =\frac{1}{\sigma_F^2}\int_0^T\bigl(1-\rho(r(t))^2\bigr) \E[\beta(t)^2]\ud t. 
\end{multline} 
The integrand is nonnegative. On the positive-measure subset $\{t\in S:\E[\beta(t)^2]>0\}$, we have $r(t)>0$, hence $\rho(r(t))^2<1$ by strict monotonicity, and $\E[\beta(t)^2]>0$; therefore $\Delta \FI>0$.
\end{proof}

Proposition~\ref{prop:tracking}(i) is an operational converse: provided the blue team can keep up with the red team's nominal HJ path, the information limit of Section~\ref{sec:identifiability} is achievable. Part~(ii) is a strict-gap statement: any deviation from exact tracking at times where the prize density $\E[\beta(t)^2]$ is positive, under a strictly decreasing detection function, loses information. The loss is given exactly by the weighted integral~\eqref{eq:tracking_gap}. For speed-limited blue teams that cannot track $\xFvec$ exactly, the optimization~\eqref{eq:sweep_design} maximizes the achievable information, equivalently minimizing this weighted loss over feasible sweeps; the experiments of Section~\ref{subsec:sweep_numerics} quantify the loss in a representative configuration.

\subsection{Numerical illustration}
\label{subsec:sweep_numerics}

We solve~\eqref{eq:sweep_design} in the harbor vortex of Section~\ref{subsec:exp_vortex} at $M_0=1$, with Gaussian detection function $\rho(r)=e^{-r^2/(2L^2)}$, $L=200$ m, and speed bound $v_L=3.0$ m/s. The kernel $\E[\beta(t)^2]$ is precomputed from the moment ODEs~\eqref{eq:mean_ode}--\eqref{eq:variance_ode} along the nominal HJ path. To ensure the red team remains in motion during the prize window (rather than parking at the destination late in the horizon), we shorten the observation horizon to $T=500$ s, below the red team's HJ travel time of approximately $556$ s. The prize density is then broadly distributed over $t\in[100,400]$ s and moves with the red team along the HJ trajectory.

The blue team's initial position is $\mathbf x_0^L=(-400,400)$ m, offset from the red start $\mathbf x_0^F=(-600,0)$ m by approximately $450$ m. The blue sweep $\xLvec$ is parameterized as a two-dimensional cubic spline with eight control points uniformly spaced in $[0,T]$ and pinned to $\mathbf x_0^L$ at $t=0$; \eqref{eq:sweep_design} is solved by L-BFGS-B with a quadratic penalty on speed-bound violations. We compare two sweeps against the direct-observation upper bound:
\begin{enumerate}[label=(\arabic*)]
\item \emph{Optimal sweep}: solution of \eqref{eq:sweep_design}.
\item \emph{Stationary}: $\xLvec(t)\equiv \mathbf x_0^L$, the blue team remains at its initial position throughout the observation.
\end{enumerate}

\begin{figure}[htbp]
  \centering
  \includegraphics[width=0.78\linewidth]{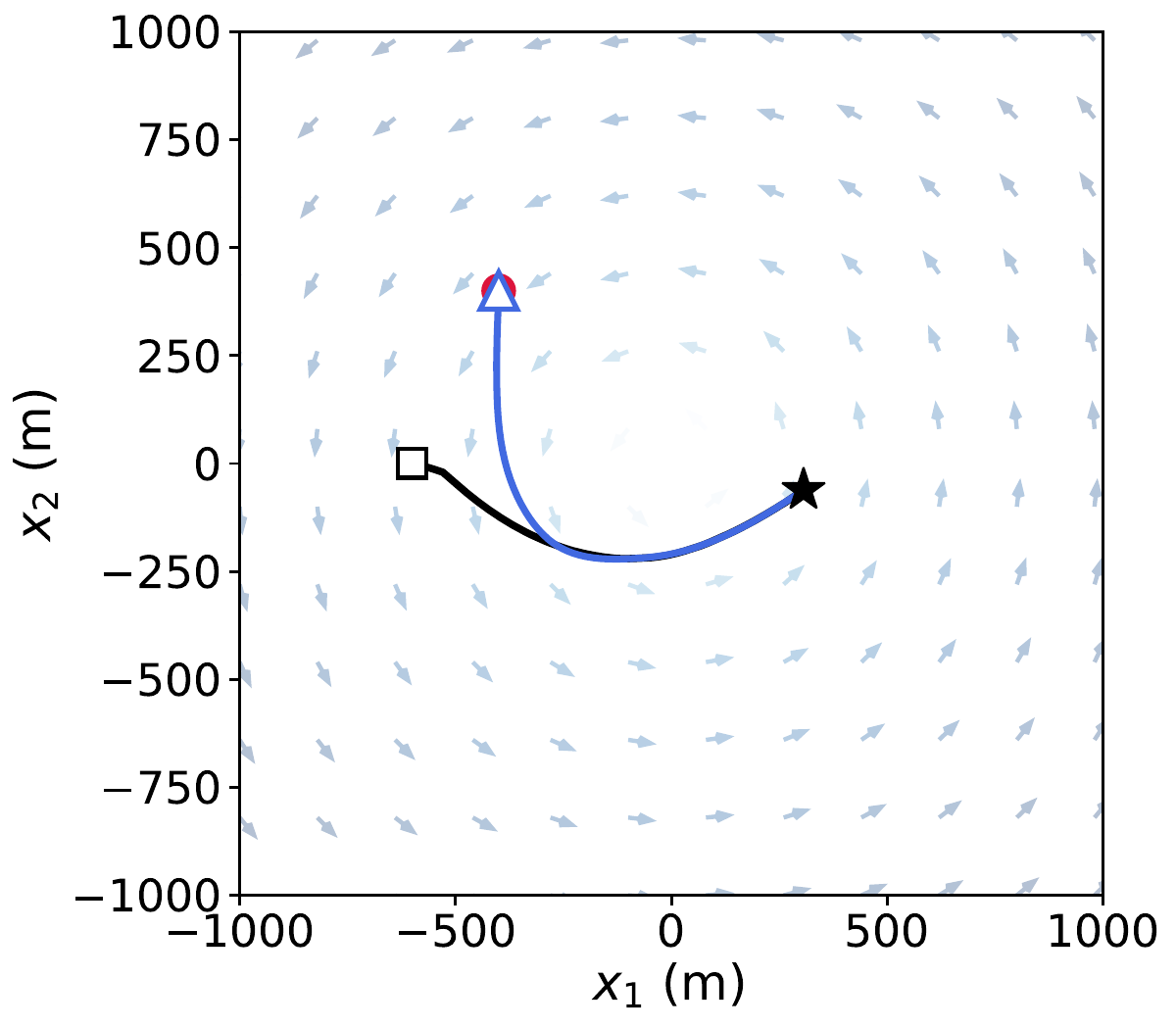}
  \caption{Sweep strategies in the harbor vortex \eqref{eq:vortex} at $M_0=1$, $T=500$ s. Black solid: nominal HJ path (white square: red start; black star: red position at $t=T$). Blue solid: optimal sweep~\eqref{eq:sweep_design}, $v_L=3.0$ m/s, from the white triangle $\mathbf x_0^L=(-400,400)$ m. Red dot: stationary baseline. The optimal sweep descends to intercept the moving HJ path and tracks the red team through the prize-density window.}
  \label{fig:sweep_compare}
\end{figure}

\begin{figure}[htbp]
  \centering
  \includegraphics[width=0.78\linewidth]{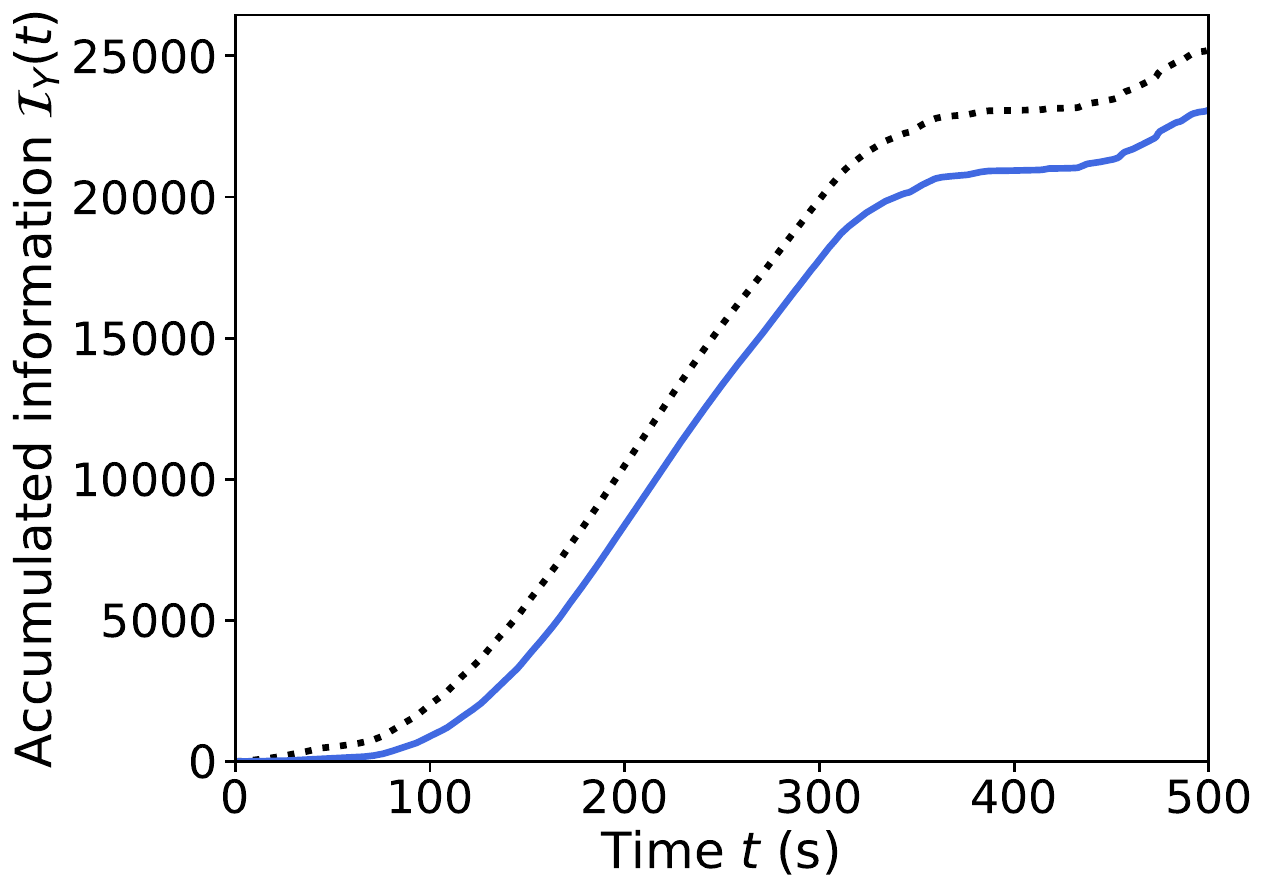}
  \caption{Accumulated information for the optimal sweep of Figure~\ref{fig:sweep_compare} $\mathcal{I}_Y(t) =\sigma_F^{-2}\!\int_0^t\!\rho(r(s))^2\E[\beta(s)^2]\ud s$. Black dotted: direct observation $\rho\equiv1$ (upper bound); blue solid: optimal sweep, recovering $91.5\%$. The stationary baseline (omitted) accumulates only $1.6\%$. The prize density is spread over $t\in[100,400]$ s, so the optimal sweep must keep moving.}
  \label{fig:sweep_FI}
\end{figure}

Two observations. First, the optimal sweep is genuinely a moving trajectory: blue must intercept the HJ path and follow the red team through the prize window. The recovery of $91.5\%$ falls short of the perfect-tracking limit of Proposition~\ref{prop:tracking}(i) because $v_L$ and the initial offset do not allow exact tracking over the full window; information is lost to the transit phase and residual tracking error. Second, the stationary baseline recovers only $1.6\%$: the red team passes through its detection range only briefly, and most of the prize density accumulates when the red team is elsewhere, as predicted by Proposition~\ref{prop:tracking}(ii).

\section{Conclusion}
\label{sec:conclusion}

We have developed a continuous-time likelihood theory for adversarial intent inference when the observed agent follows HJ-optimal paths in a current field. The framework yields a closed-form scalar MLE, a constant-memory multi-period estimator that is asymptotically efficient, and a geometric characterization of the underlying Fisher information.

Within the affine Gaussian model, inference quality is determined by the combined score kernel $\beta=\gtil+h\,\delta X_t$: the multiplicative component $h$ captures drift sensitivity, with leading geometric contribution the current-Hessian/Jacobi contraction, while the additive components $g$ and $\gamma$ capture, respectively, the policy-mean response to the HJ reference and the observation-model mismatch induced by the reference's own sensitivity to $M$. The HJ/current geometry thus enters the inference problem twice: it shapes the rational agent's reference path, and it drives the sensitivities that determine what can be recovered from observations. The result is a computable map from current-field geometry to destination inferability. Numerical experiments in vortex and channel-shear fields show that the combined kernel predicts both the relative ordering of HJ-optimal and dead-reckoning references and the variance scale of the resulting estimators.

Three extensions are natural. First, the modeling approximations of Remark~\ref{rem:AF_origin} and Assumption~\ref{ass:affine} can be relaxed; the geometric structure identified in Section~\ref{sec:identifiability} points to the right nonlinear objects for such an analysis. Second, the scalar parameterization $\mathbf x^F_T(M)=\Gamma(M)$ assumes a known one-dimensional destination curve and reduces inference to estimating a single scalar; in operationally realistic scenarios the destination may lie anywhere in a region or be one of a discrete set of candidates. The extension is analytically direct: the combined kernel $\beta(t)$ becomes vector-valued (one component per parameter), the Fisher information becomes a matrix, and identifiability becomes a rank condition. Discrete model selection over a finite candidate set reduces to pairwise likelihood ratios in the same framework. The numerical and computational considerations in this multi-parameter setting, in particular conjugate points and matrix-rank degeneracies of the Jacobi field, are substantive and worth developing separately. Third, the active sweep-design framework of Section~\ref{sec:sweep_design} treats the blue team's design problem in a fixed nominal model; iterating the design over a posterior on $M$ turns~\eqref{eq:sweep_design} into a Bayesian experiment-design problem, and coupling it to the sparse multi-period estimator of Section~\ref{subsec:sparse} would yield adaptive sweep-design protocols, with the sweep updated between periods as the posterior refines. Both the Bayesian-design and adaptive extensions preserve the core structure of HJ-induced kernel times sweep-dependent coverage, and are natural directions for future work.

\noindent{\bf Acknowledgements.} The authors thank Dr. Jeremy Brandman for helpful discussions and for detailed feedback on an earlier draft of this manuscript.

\noindent {\bf Declaration of generative AI technologies in the manuscript preparation process.} During the preparation of this work, the authors used ChatGPT for English language polishing and to improve the clarity of presentation. The authors reviewed and edited the content as needed and take full responsibility for the content of the submitted manuscript.

\bibliographystyle{plain}
\bibliography{reference}

@inproceedings{brandman2023globally,
  title={Globally time-optimal path planning for unmanned underwater vehicles in three-dimensional current fields using {H}amilton-{J}acobi partial differential equations},
  author={Brandman, Jeremy and Olson, Colin},
  booktitle={OCEANS 2023-MTS/IEEE US Gulf Coast},
  pages={1--10},
  year={2023},
  organization={IEEE}
}

@inproceedings{greeley2023reinforcement,
  title={Reinforcement learning for improved guidance and power management of unmanned underwater vehicles},
  author={Greeley, Brian and Brandman, Jeremy and Olson, Colin},
  booktitle={OCEANS 2023-MTS/IEEE US Gulf Coast},
  pages={1--10},
  year={2023},
  organization={IEEE}
}

@article{hu2024strategic,
  author  = {Hu, Ruimeng and Ralston, Daniel and Yang, Xu and Zhou, Haosheng},
  title   = {Strategic Inference in {Stackelberg} Games: Optimal Control
             for Revealing Adversary Intent},
  journal = {arXiv:2510.05641},
  year    = {2025},
}

@book{fossen2021handbook,
  title={Handbook of marine craft hydrodynamics and motion control},
  author={Fossen, Thor I},
  year={2011},
  publisher={John wiley \& Sons}
}

@article{lolla2014time,
  title={Time-optimal path planning in dynamic flows using level set equations: theory and schemes},
  author={Lolla, Tapovan and Lermusiaux, Pierre FJ and Ueckermann, Mattheus P and Haley Jr, Patrick J},
  journal={Ocean Dynamics},
  volume={64},
  number={10},
  pages={1373--1397},
  year={2014},
  publisher={Springer}
}

@article{zhao2004fast,
  title={A fast sweeping method for eikonal equations},
  author={Zhao, Hongkai},
  journal={Mathematics of Computation},
  volume={74},
  number={250},
  pages={603--627},
  year={2005}
}

@book{evans2002pde,
  title={Partial differential equations},
  author={Evans, Lawrence C},
  volume={19},
  year={2022},
  publisher={American Mathematical Society}
}

@book{basar1998dynamic,
  title={Dynamic noncooperative game theory},
  author={Ba{\c{s}}ar, Tamer and Olsder, Geert Jan},
  year={1998},
  publisher={SIAM}
}

@article{yong2002leader,
  title={A leader-follower stochastic linear quadratic differential game},
  author={Yong, Jiongmin},
  journal={SIAM Journal on Control and Optimization},
  volume={41},
  number={4},
  pages={1015--1041},
  year={2002},
  publisher={SIAM}
}

@book{kutoyants2013statistical,
  title={Statistical inference for ergodic diffusion processes},
  author={Kutoyants, Yury A},
  year={2004},
  publisher={Springer}
}

@book{liptser1977statistics,
  title={Statistics of random processes II: Applications},
  author={Liptser, Robert S. and Shiryaev, Albert N},
  volume={6},
  year={2013},
  publisher={Springer Science \& Business Media}
}

@article{schulman2017proximal,
  title={Proximal policy optimization algorithms},
  author={Schulman, John and Wolski, Filip and Dhariwal, Prafulla and Radford, Alec and Klimov, Oleg},
  journal={arXiv:1707.06347},
  year={2017}
}

@article{xu2026traveltime,
  title={Travel-time tomography from mean field game dynamics},
  author={Xu, Longqiang and Yin, Weishi and Liu, Hongyu},
  journal={arXiv:2605.15602},
  year={2026}
}

@book{sutton2018reinforcement,
  title={Reinforcement learning: An introduction},
  author={Sutton, Richard S and Barto, Andrew G and others},
  volume={1},
  number={1},
  year={1998},
  publisher={MIT press Cambridge}
}

@article{jaynes1957information,
  title={Information theory and statistical mechanics},
  author={Jaynes, Edwin T},
  journal={Physical Review},
  volume={106},
  number={4},
  pages={620},
  year={1957},
  publisher={APS}
}

@book{pukelsheim2006optimal,
  title={Optimal design of experiments},
  author={Pukelsheim, Friedrich},
  year={2006},
  publisher={SIAM}
}

@article{feldbaum1960dual,
  title={Dual control theory. {I}},
  author={Feldbaum, Aleksandr A.},
  journal={Avtomatika i Telemekhanika},
  volume={21},
  number={9},
  pages={1240--1249},
  year={1960}
}

@article{osher1988fronts,
  title={Fronts propagating with curvature-dependent speed: {A}lgorithms based on {H}amilton-{J}acobi formulations},
  author={Osher, Stanley and Sethian, James A},
  journal={Journal of Computational Physics},
  volume={79},
  number={1},
  pages={12--49},
  year={1988},
  publisher={Elsevier}
}

@inproceedings{garau2005path,
  title={Path planning of autonomous underwater vehicles in current fields with complex spatial variability: an {A}* approach},
  author={Garau, Bartolome and Alvarez, Alberto and Oliver, Gabriel},
  booktitle={Proceedings of the 2005 IEEE international conference on robotics and automation},
  pages={194--198},
  year={2005},
  organization={IEEE}
}

@article{petres2007path,
  title={Path planning for autonomous underwater vehicles},
  author={Petres, Cl\'{e}ment and Pailhas, Yan and Patron, Pedro and Petillot, Yvan and Evans, Jonathan and Lane, David},
  journal={IEEE Transactions on Robotics},
  volume={23},
  number={2},
  pages={331--341},
  year={2007},
  publisher={IEEE}
}

@article{ward2023active,
  title={Active inverse reward design},
  author={Mindermann, S{\"o}ren and Shah, Rohin and Gleave, Adam and Hadfield-Menell, Dylan},
  journal={arXiv:1809.03060},
  year={2018}
}

@article{lekkas2014guidance,
  title={Guidance and path-planning systems for autonomous vehicles},
  author={Lekkas, Anastasios M},
  journal={Norwegian University of Science and Technology: Trondheim, Norway},
  year={2014}
}

@book{karatzas1998brownian,
  title={Brownian motion and stochastic calculus},
  author={Karatzas, Ioannis and Shreve, Steven},
  year={2014},
  publisher={Springer}
}

@article{bergemann2019information,
  title={Information design: A unified perspective},
  author={Bergemann, Dirk and Morris, Stephen},
  journal={Journal of Economic Literature},
  volume={57},
  number={1},
  pages={44--95},
  year={2019},
  publisher={American Economic Association 2014 Broadway, Suite 305, Nashville, TN 37203-2425}
}

@article{mesbah2018stochastic,
  title={Stochastic model predictive control with active uncertainty learning: A survey on dual control},
  author={Mesbah, Ali},
  journal={Annual Reviews in Control},
  volume={45},
  pages={107--117},
  year={2018},
  publisher={Elsevier}
}

@article{crandall1983viscosity,
  title={Viscosity solutions of {H}amilton-{J}acobi equations},
  author={Crandall, Michael G. and Lions, Pierre-Louis},
  journal={Transactions of the American Mathematical Society},
  volume={277},
  number={1},
  pages={1--42},
  year={1983}
}

@book{bardi1997optimal,
  title={Optimal control and viscosity solutions of {H}amilton-{J}acobi-{B}ellman equations},
  author={Bardi, Martino and Capuzzo-Dolcetta, Italo},
  volume={12},
  year={1997},
  publisher={Springer}
}

@book{bar2004estimation,
  title={Estimation with applications to tracking and navigation: theory algorithms and software},
  author={Bar-Shalom, Yaakov and Li, X. Rong and Kirubarajan, Thiagalingam},
  year={2001},
  publisher={John Wiley \& Sons}
}

@article{ljung1999system,
  title={System identification: {T}heory for the user},
  author={Ljung, Lennart},
  journal={PTR Prentice Hall, Upper Saddle River, NJ},
  year={1999}
}

@article{astrom1971system,
  title={System identification—a survey},
  author={{\AA}str{\"o}m, Karl Johan and Eykhoff, Peter},
  journal={Automatica},
  volume={7},
  number={2},
  pages={123--162},
  year={1971},
  publisher={Elsevier}
}

@article{shoukry2015secure,
  title={Secure state estimation for cyber-physical systems under sensor attacks: {A} satisfiability modulo theory approach},
  author={Shoukry, Yasser and Nuzzo, Pierluigi and Puggelli, Alberto and Sangiovanni-Vincentelli, Alberto L and Seshia, Sanjit A and Tabuada, Paulo},
  journal={IEEE Transactions on Automatic Control},
  volume={62},
  number={10},
  pages={4917--4932},
  year={2017},
  publisher={IEEE}
}

@article{pajic2017attack,
  title={Attack-resilient state estimation for noisy dynamical systems},
  author={Pajic, Miroslav and Lee, Insup and Pappas, George J},
  journal={IEEE Transactions on Control of Network Systems},
  volume={4},
  number={1},
  pages={82--92},
  year={2017},
  publisher={IEEE}
}

@inproceedings{ward2025active,
  title={Active inverse learning in {S}tackelberg trajectory games},
  author={Ward, William and Yu, Yue and Levy, Jacob and Mehr, Negar and Fridovich-Keil, David and Topcu, Ufuk},
  booktitle={2025 American Control Conference (ACC)},
  pages={1547--1553},
  year={2025},
  organization={IEEE}
}

@inproceedings{zhou2025integrating,
  title     = {Integrating Sequential Hypothesis Testing into Adversarial Games: A Sun Zi-Inspired Framework},
  author    = {Zhou, Haosheng and Ralston, Daniel and Yang, Xu and
               Hu, Ruimeng},
  booktitle = {2025 {IEEE} 64th Conference on Decision and Control ({CDC})},
  pages     = {4540--4546},
  year      = {2025},
  organization = {IEEE},
}

@article{zhou2025adversarial,
  title   = {Adversarial Decision-Making in Partially Observable Multi-Agent Systems: A Sequential Hypothesis Testing Approach},
  author  = {Zhou, Haosheng and Ralston, Daniel and Yang, Xu and Hu, Ruimeng},
  journal = {{IEEE} Transactions on Control of Network Systems},
  note    = {to appear, arXiv:2509.03727},
  year    = {2025},
}

@article{ralston2026information,
  title   = {Information Revelation and Alignment Faking in Stochastic Differential Games},
  author  = {Ralston, Daniel and Yang, Xu and Hu, Ruimeng},
  journal = {arXiv:2603.17197},
  year    = {2026},
}

@article{kim2026deception,
  title   = {Deception in Linear-Quadratic Control},
  author  = {Kim, Yerin and Zhou, Haosheng and Benvenuti, Alexander and Hu, Ruimeng and Hale, Matthew},
  journal = {arXiv:2604.00227},
  year    = {2026},
}

\end{document}